\documentclass[09pt]{article}

\usepackage{amsmath}
\usepackage{url}
\usepackage{amsthm}
\usepackage{amsfonts}
\usepackage{hyperref}
\usepackage{accents}
\usepackage{xcolor}

\usepackage{enumitem}
\usepackage{graphics}
\usepackage{wrapfig}
\usepackage{graphicx}
\newtheorem{theorem}{Theorem}[section]
\newtheorem*{theorem*}{Theorem}
\newtheorem{lemma}[theorem]{Lemma}

\newtheorem{Claim}{Claim}

\newtheorem{prop}{Proposition}[section]
\newtheorem{remark}[prop]{Remark}
\newtheorem{definition}[prop]{Definition}

\makeatletter \@addtoreset{equation}{section} \makeatother

\newcommand{\chapter}

\newcommand{\R}{\mathbb{R}}

\newcommand{\circA}{\accentset{\circ}{A}}

\def\<{\langle}
\def\>{\rangle}

\usepackage{ulem} 

\begin{document}
\title{
On the Morse index of free-boundary \\ 
CMC hypersurfaces in the upper hemisphere 
}
\author{CRÍSIA DE OLIVEIRA}
\maketitle

\begin{abstract}
\noindent We prove results for free-boundary hypersurfaces in the upper unit hemisphere $\mathbb{S}^{n+1}_{+}$ 
of $\R^{n+2}$. First we show that if the norm squared of the second fundamental form  
is constant, the Morse index of a free-boundary minimal hypersurface $\Sigma\subset \mathbb{S}^{n+1}_{+}$ equals: $1$ if $\Sigma$ is a totally geodesic equator, $n+1$ if $\Sigma$ is half of the Clifford torus, or it is at least $2(n+1)$ when $\Sigma$ is not totally geodesic. Next we prove an estimate for the first eigenvalue $\lambda_1$ of the second variation's Jacobi operator, and show that $\lambda_1 \leq -2n$ if $\Sigma$ is not totally geodesic, with equality  iff $\Sigma$ is half of the minimal Clifford torus. Furthermore,  $\lambda_1 = -n$ iff $\Sigma$ is totally geodesic. Finally,  if $\Sigma$ is not totally umbilical the Morse index is at least $n+1$, with equality precisely when $\Sigma$ is the upper $\mathrm{H}$-torus. For totally umbilical hypersurfaces the Morse index is $1$. We also prove an upper bound for the first eigenvalue of free-boundary $\mathrm{CMC}$ hypersurfaces, 
where equality corresponds to totally umbilical hypersurfaces.
\end{abstract}

\section{Introduction}
Let $x_i$ denote the stardard coordinates on $\R^{n+2}$ and call
$$\mathbb{S}^{n+1}_{+} = \mathbb{S}^{n+1} \cap \{x_{n+2} \geq 0\}$$ 
the upper half of the unit sphere $\mathbb{S}^{n+1}$ in $\mathbb{R}^{n+2}$. Consider an immersion $x: \Sigma \longrightarrow \mathbb{S}^{n+1}_{+}$, with non-empty boundary $\partial \Sigma$ and constant mean curvature $(\mathrm{CMC})$. We call $x$ a free-boundary hypersurface if has constant contact angle $\theta = \frac{\pi}{2}$. Free-boundary hypersurfaces are known to be solutions to the variational problem given by the area functiona
\begin{align}
    \mathcal{A}(t) = \int_{\Sigma}d\Sigma_{t}.
\end{align}
More precisely, consider a variation of the hypersurface $x: \Sigma \longrightarrow \mathbb{S}^{n+1}_{+}$. This is, for any $t \in (-\epsilon, \epsilon)$, a family
\begin{align*}
    F_t: \Sigma \longrightarrow \mathbb{S}^{n+1}_{+}
\end{align*}
of immersions with $F_{t}(\mathrm{int}(\Sigma)) \subset \mathrm{int}(\mathbb{S}^{n+1}_{+})$, $F_t(\partial \Sigma) \subset \partial \mathbb{S}^{n+1}_{+}$ and $F_0 = x$. Free-boundary $\mathrm{CMC}$ hypersurfaces with non-empty boundary are characterized as the  critical points if we restrict to  volume preserving variations, i.e. solutions to
\begin{align}\label{Eq. 1.2}
    \mathcal{A}'(0) = \frac{d \mathcal{A}(t)}{dt}_{|t=0} = 0.
\end{align}

From a Morse-theoretical perspective, it is crucial to examine the area functional's second variation at the critical points, which entails studying a symmetric bilinear form in an appropriate function space. The second variation is given by the following 
formula, for smooth maps $f$ and $g$,
\begin{align}
    \mathrm{Q}(f,g) = -\int_{\Sigma}(\langle \nabla f, \nabla g \rangle - pfg)d\mu + \int_{\partial \Sigma}\mathrm{q}fg ds,
\end{align}
where the functions $p$ and $\mathrm{q}$ are determined by the geometry of $\Sigma$, and typically depend on the specific variational problem under consideration.

   The Morse index of $\Sigma$, denoted by $\mathrm{MI}(\Sigma)$, is the maximum dimension of any subspace $V$ of $C^{\infty}(\Sigma)$ on which $Q$ is negative definite.

In \cite[Theorem $5.1.1$]{Simons}, Simons proved that if $\Sigma$ is minimal, then $\mathrm{MI}(\Sigma) \geq 1$, and equality holds if and only if $\Sigma$ is a totally geodesic equator $\mathbb{S}^n \subset \mathbb{S}^{n+1}$. Later, Urbano \cite{Urbano} showed that if $n=2$, $\mathrm{MI}(\Sigma) \geq 5$ when $\Sigma$ is not totally geodesic. He characterized the Clifford torus as the only compact, non-totally geodesic minimal surface in $\mathbb{S}^{3}$ whose index is $5$. Urbano's result is fundamental in the study of minimal surfaces in the sphere, and was recently used by Marques and Neves to solve the  Willmore conjecture (see \cite{Codá}). 
\bigbreak

\noindent One may ask whether similar phenomena happen for free-boundary minimal hypersurfaces in the upper hemisphere $\mathbb{S}^{n+1}_{+}$. When 
the immersion $x: \Sigma \longrightarrow \mathbb{S}^{n+1}_{+}$ is minimal, such hypersurfaces are volume-critical among all deformations that  preserve the boundary $\partial \Sigma \subset \partial \mathbb{S}^{n+1}_{+}$. 
In the free-boundary case, the quadratic form induced by the Jacobi operator is

\begin{align*}
     Q(f) = -\int_{\Sigma}f Jgf\mu + \int_{\partial \Sigma}f(\nabla_{\eta}f-\mathrm{q}f) ds, \forall f\in V,
\end{align*}
where $\mathrm{q} = \mathrm{A}^{\partial \mathbb{S}^{n+1}_{+}}$ is the second fundamental form of $\partial \mathbb{S}^{n+1}_{+}$, $\eta$ is a conormal unit vector of  $\partial \Sigma$ and $ \nabla_{\eta}f = \langle \nabla f, \eta \rangle$. 
\bigbreak

\noindent The present paper's first result concerns free-boundary minimal hypersurfaces in the upper hemisphere, and goes as follows.

 \begin{theorem}
    Let $x: \Sigma \longrightarrow \mathbb{S}^{n+1}_{+}$ be a free-boundary minimal hypersurface and $\eta = e_{n+2}$ the conormal unit vector of $\partial \Sigma$. Suppose that $|\mathrm{A}|$ is constant. Then there are three possibilities: 
    \begin{enumerate}[label=(\roman*)]
    \item   $\mathrm{MI}(\Sigma) = 1$ and  $\Sigma=\mathbb{S}^{n}_{+} \subset \mathbb{S}^{n+1}_{+}$;
    \item  $\mathrm{MI}(\Sigma) = n+1$ and $\Sigma$ is half of a minimal Clifford torus:
    \begin{equation}\label{eq:1/2-Clifford}
        \Sigma=\mathbb{S}^{k}_{+}\left(\sqrt{\frac{k}{n}}\right) \times \mathbb{S}^{n-k}_{+}\left(\sqrt{\frac{n-k}{n}}\right);  
    \end{equation}
    \item  $\mathrm{MI}(\Sigma) \geq 2(n+1)$.
    \end{enumerate}
\end{theorem}
\bigbreak

\noindent Next we shall take a closer look at the first eigenvalue of the Jacobi operator $J$ of the second variation. For a compact minimal hypersurface in $\mathbb{S}^{n+1}$, Alías \cite[Theorem $3$]{Alías 1} proved that $\lambda_1 = -n$ if and only if $x: \Sigma \longrightarrow \mathbb{S}^{n+1}$ is totally geodesic. He further showed that $\lambda_1 \leq -2n$ when $x: \Sigma \longrightarrow \mathbb{S}^{n+1}$ is not totally geodesic, with equality precisely when $x$ gives the Clifford torus $\mathbb{S}^{k}\left(\sqrt{\frac{k}{n}} \right) \times \mathbb{S}^{n-k}\left(\sqrt{\frac{n-k}{n}} \right)$.

 In the free-boundary case, let $\lambda_1$ indicate the first eigenvalue of the Neumann problem
\begin{align*}
\begin{cases}
        &Jf - \lambda f = 0 \ \textrm{on} \ \Sigma\\
        &\nabla_{\eta}f = 0 \ \textrm{on} \  \partial \Sigma.
        \end{cases}
    \end{align*}
As in our case $\mathrm{q} = 0$, $\lambda_1$ becomes
    \begin{align}\label{Eq.1.1}
         \lambda_1 = \displaystyle\inf \left\{\frac{-\displaystyle\int_{\Sigma}fJf d\mu}{\displaystyle\int_{\Sigma}f^2 d\mu}\colon f \in C^{\infty}(\Sigma) \setminus \{0\}\right\}.
    \end{align}
    We shall prove the following.
    
\begin{theorem}
 Let $x: \Sigma \longrightarrow \mathbb{S}^{n+1}_{+}$ be a free-boundary minimal hypersurface such that $\nabla_{\eta}|\mathrm{A}|^2 = 0$. Then either
 \begin{enumerate}[label=(\roman*)]
 \item  $\lambda_1 \leq -2n $, with equality if and only if $\Sigma$ is \eqref{eq:1/2-Clifford} or
 \item $\lambda_1 = -n$ and $\Sigma$ is totally geodesic.
 \end{enumerate}
\end{theorem}
As a matter of fact, we are able to generalize the above to the free-boundary $\mathrm{CMC}$ case with respect to weak Morse index of $\Sigma$, we denoted by $\mathrm{MI}_{W}(\Sigma)$.
\begin{theorem}
    Let $x: \Sigma \longrightarrow \mathbb{S}^{n+1}_{+}$ be a free-boundary hypersurface with constant mean curvature $\mathrm{H}>0$, and call  $\eta = e_{n+2}$ a conormal unit vector of $\partial \Sigma$. If $|\mathrm{A}|$ is constant  then 
    \begin{itemize}
        \item i. either $\mathrm{MI}_{W}(\Sigma) = 0$ and $ \mathbb{S}^{n}_{+}(r) \subset \mathbb{S}^{n+1}_{+}$,
        \item ii. or $\mathrm{MI}_{W}(\Sigma) \geq n+1$, with equality if and only if $\Sigma$ is upper $\mathrm{H}$- Clifford torus $\mathbb{S}^{k}(r)_{+}\times \mathbb{S}^{n-k}(\sqrt{1-r^2})$ with radius $\sqrt{\frac{k}{n+2}} \leq r \leq \sqrt{\frac{k+2}{n+2}}$. 
    \end{itemize}
    
\end{theorem}
Alías and collaborators studied the first eigenvalue of the Jacobi operator $J$ on compact hypersurfaces with CMC $\mathrm{H}>0$, obtaining the following \cite[Theorem $2.2$]{Alías 4}.
\begin{theorem*}
    [Alías-Brasil-Perdomo]
    Let $\Sigma$ be an orientable, compact minimal hypersurface immersed in 
    $\mathbb{S}^{n+1}$ and let $\lambda_1$ indicate the first eigenvalue of its stability operator $J = \Delta + |\mathrm{\circA}|^2 + n(1+ \mathrm{H}^2)$. Then
    \begin{enumerate}[label=(\roman*)]
        \item either $\lambda_1 = -n(1+ \mathrm{H}^2)$ (and $\Sigma$ is totally umbilical),
        \item or $\lambda_1 \leq -2n(1+ \mathrm{H}^2) + \frac{n(n-2)}{\sqrt{n(n-1)}}|\mathrm{H}|\max|\mathrm{\circA}|$.
    \end{enumerate}
    Moreover, $\lambda_1 = -2n(1+ \mathrm{H}^2) + \frac{n(n-2)}{\sqrt{n(n-1)}}|\mathrm{H}|\max|\mathrm{\circA}|$ if and only if
    \begin{enumerate}[label=(\roman*)]
        \item $n=2$ and $\Sigma$ is an $\mathrm{H}(r)$-torus $\mathbb{S}^1(r) \times \mathbb{S}^1(\sqrt{1-r^2})$,
         \item $n\geq 3$ and $\Sigma$ is an $\mathrm{H}(r)$-torus $\mathbb{S}^{n-1}(r) \times \mathbb{S}^1(\sqrt{1-r^2})$, with $r^2 \leq \frac{n-1}{n}$.
    \end{enumerate}
\end{theorem*}
This result should be compared to the ensuing estimate for the first eigenvalue of free-boundary CMC hypersurfaces in $\mathbb{S}^{n+1}_{+}$.
\begin{theorem}
    Let $x: \Sigma \longrightarrow \mathbb{S}^{n+1}_{+}$ be a free-boundary hypersurface with constant mean curvature $\mathrm{H}>0$ and let $\lambda_1$ denote the first eigenvalue of its Jacobi operator $J = \Delta + |\mathrm{\circA}|^2 + n\left(1+\displaystyle\frac{\mathrm{H}^2}{n^2}\right)$. Assume that $\nabla_{\eta}|\mathrm{\circA}|^2 = 0$. Then either
   \begin{enumerate}[label=(\roman*)]
       \item $\lambda_1 = -n\left(1+\displaystyle\frac{\mathrm{H}^2}{n^2}\right)$ and $\Sigma= \mathbb{S}^{n}_{+}(r) \subset \mathbb{S}^{n+1}_{+}$,
       \item or  $\lambda_1 \leq  - 2n\left(1+\displaystyle\frac{\mathrm{H}^2}{n^2}\right) +  \displaystyle\frac{\mathrm{H}(n-2)}{\sqrt{n(n-1)}}\frac{\displaystyle\int_{\Sigma}|\mathrm{\circA}|^3 d\mu}{\displaystyle\int_{\Sigma}|\mathrm{\circA}|^2 d\mu}$, with equality if and only if $\Sigma$ is half of the $\mathrm{H}$-torus.
   \end{enumerate}
\end{theorem}
\vspace{.5cm}

\noindent The paper is organized as follows. After laying out the backgrounds in section \ref{Preliminares}, in section \ref{Section 3} we obtain valuable information on the Morse index for free-boundary minimal hypersurfaces in $\mathbb{S}^{n+1}_{+}$, plus an estimate for the first eigenvalue of the Jacobi operator. 
Section \ref{Section 4} is devoted to proving upper bounds for the Morse index of free-boundary minimal hypersurfaces with positive constant mean curvature in $\mathbb{S}^{n+1}_{+}$ and for the Jacobi operator's first eigenvalue.
\bigbreak    

\noindent {\it Ackowledgements.}
This article is part of the author's PhD thesis. She would like to extend her heartfelt thanks to Detag Zhou (UFF), Luis Al\'ias (University of Murcia) and Lucas Ambrozio (IMPA) for the many enlightening conversations. I would also like to thank professors Marcos Cavalcante (UFAL), Simon Chiossi (UFF) and Helton Leal (UFF) for their wonderful suggestions. The author was supported by CAPES 

\section{Preliminaries}\label{Preliminares}

We shall begin by setting out the notation used throughout, and recall some elementary facts. We denote by  $\mathbb{S}^{n+1}$ the unit sphere in $\mathbb{R}^{n+2}$ and by $\mathbb{S}^{n+1}_{+} = \mathbb{S}^{n+1}\cap \{x_{n+2} \geq 0\}$ the upper hemisphere. Let $$x: \Sigma \longrightarrow \mathbb{S}^{n+1}_{+}$$ be the immersion of an $n$-dimensional orientable Riemannian hypersurface, possibly with boundary $\partial \Sigma$. Let us further denote
a unit normal vector of the hypersurface $\Sigma$ by $\nu$. We will refer to the above immersion, or to $\Sigma$ directly, as {\sl free boundary} if $x(\Sigma)$ meets $\partial \mathbb{S}^{n+1}_{+}$ orthogonally,  $x(\mathrm{int} \Sigma) \subset \mathrm{int}\mathbb{S}^{n+1}_{+}$ and $\partial \Sigma \subset \partial \mathbb{S}^{n+1}_{+}$. The shape operator $\mathrm{S}_{\nu}: \mathcal{X}(\Sigma) \longrightarrow \mathcal{X}(\Sigma)$   is defined by
 \begin{align*}
     \mathrm{S}_{\nu}(X) = -(\overline{\nabla}_{X}\nu)^{T},
 \end{align*}
 where $X \in T\Sigma$, $\overline{\nabla}$ is the sphere's Levi-Civita connection 
 and $(\overline{\nabla}_{X}\nu)^{T}$ is the tangential component of $\overline{\nabla}_{X}\nu$.
The mean curvature of $\Sigma$ is given by $\mathrm{H}  =  \mathrm{trace} (\mathrm{S}_{\nu})$, and when the latter is constant we say $\Sigma$ is a \textit{CMC} surface. The second fundamental form of $\Sigma$ 
 \begin{align*}
     \mathrm{A}(X,Y)  = \langle  \overline{\nabla}_{X}Y, \nu \rangle, \quad X,Y \in T\Sigma
 \end{align*}
clearly has square norm
\begin{align*}
    | \mathrm{A}|^2 = \mathrm{trace}(\mathrm{A}^2) = \sum_{i=1}^{n} {\kappa}_i^2, 
\end{align*}
where ${\kappa}_i$ are the principal curvatures. Finally, let $\mathrm{\circA} = \mathrm{A} -n\mathrm{H}\mathrm{I}$ indicate the traceless part of $A$, where $\mathrm{I}$ is the identity operator on $\mathcal{X}(\Sigma)$. Consequently
\begin{align*}
|\mathrm{\circA}|^2 = | \mathrm{A}|^2 - \frac{ \mathrm{H}^2}{n}, \  \  n \geq 2.
\end{align*}

\begin{definition}
A \textnormal{variation} of $\Sigma$ in $\mathbb{S}^{n+1}_{+}$ is a smooth map $F: (-\epsilon, \epsilon)\times \Sigma \longrightarrow \mathbb{S}^{n+1}_{+}$, for some $\epsilon>0$, such that for each $t$, $F_t : \Sigma \longrightarrow \mathbb{S}^{n+1}_{+}$, $F_t(x) = F(t,x)$, is an immersion satisfying:
\begin{enumerate}[label=(\roman*)]
    \item $F_t(\mathrm{int}(\Sigma))\subset \mathrm{int}(\mathbb{S}^{n+1}_{+})$;
    \item $F_t(\partial \Sigma) \subset \partial \mathbb{S}^{n+1}_{+}$;
    \item $F_0 = x$.
\end{enumerate}
\end{definition}
For each $t \in (-\epsilon, \epsilon)$  consider the surface $\Sigma_t = F_t(\Sigma)$ and define the area functional $\mathcal{A}: (-\epsilon, \epsilon) \longrightarrow \mathbb{R}$ by
\begin{align*}
    \mathcal{A}(t) = \int_{\Sigma}d\mu(t),
\end{align*}
where $d\mu(t)$ is the area element of $\Sigma_t$.

The first variation of $\mathcal{A}$ measures how the area of a hypersurface changes under small variations of its immersion. 
It is well known \cite[Charpter $\mathrm{I}$, Theorem $4$]{Lawson Jr} that the first variation of  $\mathcal{A}$  is 
\begin{align}
   \mathcal{A}'(0) = \frac{d}{dt}|_{t=0}\mathcal{A}(t) = \int_{\Sigma}\mathrm{H}\langle Y, \nu \rangle d\mu + \int_{\partial \Sigma}\langle Y, \eta \rangle ds,
\end{align}
where $Y = \frac{\partial F}{\partial t}(0,x)$ is the variational field, 
$\nu$ a unit vector normal to $\Sigma$, $\eta$ the outer normal to the boundary tangent to $\Sigma$, and $ds$ is the area element of $\partial\Sigma$. 

When the hypersurface $x: \Sigma \longrightarrow \mathbb{S}^{n+1}_{+}$ with constant mean curvature $\mathrm{H}$ intersects the boundary $\partial \mathbb{S}^{n+1}_{+}$ at a constant angle $\theta = \frac{\pi}{2}$, it defines a particular class of hypersurfaces with interesting geometric properties.

\begin{definition}
    A  smooth, orientable, immersed hypersurface $x:\Sigma \longrightarrow \mathbb{S}^{n+1}_{+}$  is called a \textnormal{free-boundary} hypersurface with constant mean curvature $\mathrm{H}$ if it has constant intersection angle $\theta = \frac{\pi}{2}$. In particular, when $\mathrm{H}=0$   we say $x:\Sigma \longrightarrow \mathbb{S}^{n+1}_{+}$ is \textnormal{free-boundary minimal} hypersurface.
    \end{definition}
    In both the minimal and CMC cases, free-boundary hypersurfaces are critical values of a functional. When considered as critical points, we can show that they are local minima up to a finite-dimensional space. The dimension of this space is known as the hypersurface's stability index. In other words, the stability index of a minimal or CMC free- boundary hypesurface measures the number of essential directions along which the hypersurface fails to be area-minimizing.

   In order to understand the stability properties of critical points we need to examine how $\mathcal{A}$ varies when we change $x$. This is controlled by the second variation (see \cite[Charpter $\mathrm{I}$, Theorem $32$]{Lawson Jr}) 
    \begin{align}\label{Eq. 1.2}
        \mathcal{A}''(0) = \int_{\Sigma}fJfd\mu + \int_{\partial \Sigma}f(\nabla_{\eta}f-\mathrm{q}f) ds,
    \end{align}
    where $J : C^{\infty}(\Sigma) \longrightarrow C^{\infty}(\Sigma)$ is the Jacobi operator and
    \begin{align}\label{Eq.1.3}
       \mathrm{q} = \mathrm{A}^{\partial \mathbb{S}^{n+1}_{+}}(\overline{\nu}, \overline{\nu}),
    \end{align}
    where the last equality arises because  $\Sigma$ is free-boundary.\\

        The operator $J$ induces a quadratic form $\mathrm{Q}: C^{\infty}(\Sigma) \longrightarrow \mathbb{R}$ given by
    \begin{align}\label{Eq.1.9}
        \mathrm{Q}(f) = -\int_{\Sigma}f Jf d\mu + \int_{\partial \Sigma}f(\nabla_{\eta}f-\mathrm{A}^{\partial \mathbb{S}^{n+1}_{+}}(\overline{\nu},\overline{\nu})f) ds
    \end{align}
     and hence has an associated Morse index. For the record, we recall 
that the index of a quadratic form on a vector space is the dimension of the largest subspace on which the form is negative definite. 
Intuitively, the Morse index measures the number of independent directions along which the hypersurface fails to minimize area. Indeed, if $Q(f) < 0$ for some map $f$ then $\mathcal{A}'' < 0$, indicating that in the variational family, $\mathcal{A}(\Sigma) > \mathcal{A}(\Sigma_t)$ for small $t \neq 0$.

    The boundary condition 
\begin{align*}
    \nabla_{\eta}f = \mathrm{A}^{\partial \mathbb{S}^{n+1}_{+}}(\overline{\nu},\overline{\nu})f
\end{align*}
is an elliptic PDE for the Jacobi operator. 
The eigenvalue problem
\begin{equation}\label{P. 2.6}
\left\{
\begin{array}{lll}
 Jf - \lambda f = 0 & \textrm{on} & \Sigma\\
 \nabla_{\eta}f = \mathrm{A}^{\partial \mathbb{S}^{n+1}_{+}}(\overline{\nu},\overline{\nu})f & \textrm{on} &  \partial \Sigma,
\end{array}\right.
    \end{equation}
has an orthonormal basis $\{f_k\}_{k=1}^{\infty}$ of solutions in $L^2(d\Sigma, d\mu)$, whose eigenvalues form a divergent increasing sequence $\lambda_1 < \lambda_2 <...< \lambda_k \longrightarrow \infty$.  
    Therefore the Morse index of a free-boundary $\Sigma$ equals the number of negative eigenvalues of \eqref{P. 2.6} (see \cite{Schoen}).

Now, because $J$ is elliptic, we may express the first eigenvalue $\lambda_1$ as follows
    \begin{align*}
       \lambda_1 = \inf \left\{\frac{-\int_{\Sigma}fJf d\mu + \int_{\partial\Sigma}f(\nabla_{\eta}f-\mathrm{A}^{\partial \mathbb{S}^{n+1}_{+}}(\overline{\nu},\overline{\nu})f)ds}{\int_{\Sigma}f^2 d\mu}\colon f \in C^{\infty}(\Sigma)\setminus\{0\} \right\}.
   \end{align*}

Let us call $\nu: \Sigma \longrightarrow \mathbb{S}^{n+1}$ the Gauss map. For any  constant vector $a \in \mathbb{R}^{n+2}$
 the two smooth, real-valued maps $\langle x,a \rangle$ and $\langle \nu,a \rangle$ satisfy a number of known useful relationships, listed below for future reference.

  \begin{lemma}\label{Lemma 1.1}
Let $ x:\Sigma \longrightarrow \mathbb{S}^{n+1}_{+}$ be a hypersurface with constant mean curvature $\mathrm{H}$. 
Call $\Delta$ the sphere's Laplacian operator. 
Then
\begin{gather}
\Delta \langle x ,a \rangle =  \mathrm{H}\langle\nu ,a \rangle -n\langle x ,a \rangle, \label{Eq 1.6}\\
(\Delta + | \mathrm{A}|^2) \langle \nu,a \rangle =  \mathrm{H} \langle x,a \rangle,\label{Eq 1.7}\\
J\langle x, a \rangle = |\mathrm{A}|^2 \langle x, a \rangle +  \mathrm{H}\langle \nu, a \rangle,\label{Eq 1.8}\\
J\langle \nu, a \rangle = n\langle \nu, a \rangle +  \mathrm{H}\langle x, a \rangle, \label{Eq 1.9}\\
J(n\langle x, a \rangle -  \mathrm{H}\langle \nu, a \rangle)= (n| \mathrm{A}|^2 -  \mathrm{H}^2)\langle x, a \rangle.\label{Eq 1.10}
\end{gather}
The last can be written, equivalently, 
\begin{align}\label{Eq 1.11}
&J(-\Delta \langle x, a \rangle) = (n|\mathrm{A}|^2 -  \mathrm{H}^2)\langle x, a \rangle.
\end{align}
 \end{lemma}

For computing the area's second variation the following derivatives will turn out to be extremely valuable. 
\begin{lemma}\label{Lemma 1.2}
Let $\Sigma \longrightarrow \mathbb{S}^{n+1}_{+}$ be a hypersurface with constant mean curvature $\mathrm{H}$, 
$\eta$ a unit vector on $\partial \Sigma$, $\nabla$ and $\overline{\nabla}$ the Riemannian connections of $\Sigma$ and of the ambient $\mathbb{S}^{n+1}$. 
Then for any $a \in \mathbb{R}^{n+2}$
\begin{eqnarray}
\overline{\nabla}_{\eta}\nu & = &
\mathrm{A}(\eta, \eta)\eta\label{Eq 1.12}\\
\nabla_{\eta}\langle \nu,a \rangle & = &\mathrm{A}(\eta, \eta)\langle\eta, a \rangle \label{Eq 1.13}\\
\nabla_{\eta}\langle x,a \rangle & = &
\langle \eta, a \rangle.\label{Eq 1.14}
\end{eqnarray}
\end{lemma}
\begin{proof}
    For (\ref{Eq 1.12}) see \cite[Lemma $5$]{Li and Xiong}, so let us prove  the rest. Using the first relationship we have 
$$
      \nabla_{\eta}\langle \nu,a \rangle = \eta \langle \nu,a \rangle 
      = \langle \overline{\nabla}_{\eta}\nu, a\rangle + \langle \nu,  \overline{\nabla}_{\eta} a \rangle 
      = \mathrm{A}(\eta, \eta)\langle \eta, a\rangle + \langle \nu,  \overline{\nabla}_{\eta} a \rangle.
    $$
    Let $\nu_{\mathbb{S}^{n+1}}$ be the unit normal of $\mathbb{S}^{n+1}$. Then $\overline{\nabla}_{\eta} a = \nabla^{\mathbb{R}^{n+2}}_{\eta}a - \langle  \nabla^{\mathbb{R}^{n+2}}_{\eta}a, \nu_{\mathbb{S}^{n+1}} \rangle  \nu_{\mathbb{S}^{n+1}} = 0$, and   we obtain in one go the two covariant derivatives
    \begin{equation*}\begin{split}
        \nabla_{\eta}\langle \nu,a \rangle &= \mathrm{A}(\eta, \eta)\langle \eta, a\rangle\\
\nabla_{\eta}\langle x,a \rangle &= \eta \langle x, a \rangle 
         = \langle \overline{\nabla}_{\eta}x, a\rangle + \langle x, \overline{\nabla}_{\eta}a\rangle
        = \langle \overline{\nabla}_{\eta}x, a\rangle.
    \end{split}\end{equation*}
    Finally, from  $\nabla^{\mathbb{R}^{n+2}}_{\eta} x = \eta$ we deduce 
$         \nabla_{\eta}\langle x,a \rangle = \langle \eta, a \rangle$.
\end{proof}

\section{Free-boundary minimal hypersurfaces in $\mathbb{S}^{n+1}_{+}$}\label{Section 3}
In this section we provide a classification theorem for free-boundary minimal hypersurfaces in the upper hemisphere. We obtain that for $|\mathrm{A}| \leq n$ the hypersurfaces are either totally geodesic or they must be the upper minimal Clifford torus. This is on a par with  \cite[Theorem $6$]{Alías 1}.
\begin{lemma}\label{Lemma 3.1}
    Let $x: \Sigma \longrightarrow \mathbb{S}^{n+1}_{+}$ be a free-boundary minimal hypersurface and assume that $\nabla_{\eta}|\mathrm{A}|^2 = 0$. If $|\mathrm{A}|^2 \leq n$ then either $|\mathrm{A}| = 0$ or $|\mathrm{A}|^2 = n$. In particular, if $|\mathrm{A}|^2 = n$ then $x: \Sigma \longrightarrow \mathbb{S}^{n+1}_{+}$ is the upper minimal Clifford torus.
\end{lemma}
\begin{proof}
    We start by proving the first part claim. By Simons's equation we have
    \begin{align}\label{Eq. 3.1}
        \frac{1}{2}\Delta |\mathrm{A}|^2 = |\nabla \mathrm{A}|^2 + |\mathrm{A}|^2(n-|\mathrm{A}|^2).
    \end{align}
    Integrating the above over  $\Sigma$, we obtain
    \begin{align}\label{Eq. 3.2}
         \frac{1}{2}\int_{\Sigma}\Delta |\mathrm{A}|^2 d\mu = \int_{\Sigma}|\nabla \mathrm{A}|^2 d\mu + \int_{\Sigma}|\mathrm{A}|^2(n-|\mathrm{A}|^2) d\mu,
    \end{align}
    and Stokes' theorem on the left side gives
    \begin{align}\label{Eq. 3.3}
        \frac{1}{2}\int_{\partial \Sigma}\nabla_{\eta}\mathrm{A}|^2 ds = \int_{\Sigma}|\nabla \mathrm{A}|^2 d\mu + \int_{\Sigma}|\mathrm{A}|^2(n-|\mathrm{A}|^2) d\mu.
    \end{align}
    Using the hypothesis on equation \eqref{Eq. 3.3}, 
    \begin{align}\label{Eq. 3.4}
        \int_{\Sigma}|\nabla \mathrm{A}|^2 d\mu + \int_{\Sigma}|\mathrm{A}|^2(n-|\mathrm{A}|^2) d\mu = 0.
    \end{align}
     Since $|\nabla\mathrm{A}|^2$ and $|\mathrm{A}|^2(n-|\mathrm{A}|^2)$ are non-negative, by \eqref{Eq. 3.1} we see that
        \begin{align}\label{Eq. 3.4}
            |\nabla\mathrm{A}|^2 = 0
            \end{align}
            \begin{align}\label{Eq. 3.5}
                |\mathrm{A}|^2(n-|\mathrm{A}|^2) =0.
            \end{align}
             Now, \eqref{Eq. 3.4} and  (\ref{Eq. 3.5}) force $\mathrm{A}$ to be parallel, so either $|\mathrm{A}|=0$ or $|\mathrm{A}|^2 = n$.
              Let us prove the last part. Suppose that $|\mathrm{A}| = \sqrt{n}\neq 0$. As $|\nabla\mathrm{A}|=0$, it follows that $\Sigma$ has exactly two principal curvatures \cite[Lemma $1$]{Lawson} 
         \begin{align*}
             \pm \sqrt{\frac{n-k}{n}} \  \ \textrm{and} \  \  \mp\sqrt{\frac{n}{n-k}}
         \end{align*}
         with multiplicity $k \geq 1$ and $n-k \geq 1$, respectively. Since $\mathrm{H} = 0$ and $\Sigma$ is isoparametric, a celebrated theorem of Cartan's \cite{Cartan} ensures $\Sigma$ is a piece of  Clifford torus.
\end{proof}
\begin{remark}
     Totally geodesic hypersurfaces, and Clifford tori, satisfy the condition $\nabla_{\eta}|\mathrm{A}|^2 = 0$.  
    \end{remark}
   Now we shall obtain information about the Morse index of free-boundary minimal hypersurfaces in the upper hemisphere. 
   \begin{theorem}\label{Theorem 3.2}
    Let $x: \Sigma \longrightarrow \mathbb{S}^{n+1}_{+}$ be a free-boundary minimal hypersurface and $\eta = e_{n+2}$ the conormal unit vector of $\partial \Sigma$. Suppose that $|\mathrm{A}|$ is constant. Then there are three possibilities: 
    \begin{enumerate}[label=(\roman*)]
    \item   $\mathrm{MI}(\Sigma) = 1$ and  $\Sigma=\mathbb{S}^{n}_{+} \subset \mathbb{S}^{n+1}_{+}$;
    \item  $\mathrm{MI}(\Sigma) = n+1$ and $\Sigma$ is the upper minimal Clifford torus:
    \begin{equation}\label{eq:1/2-Clifford}
        \Sigma=\mathbb{S}^{k}_{+}\left(\sqrt{\frac{k}{n}}\right) \times \mathbb{S}^{n-k}_{+}\left(\sqrt{\frac{n-k}{n}}\right);  
    \end{equation}
    \item  $\mathrm{MI}(\Sigma) \geq 2(n+1)$.
    \end{enumerate}
\end{theorem}
\begin{proof}
    Suppose $x: \Sigma \longrightarrow \mathbb{S}^{n+1}_{+}$ is not totally geodesic. Define the vector subspaces $E_1 = \mathrm{span}\{\langle x,e_i \rangle: i=1,...,n+1\}$ and $E_2 = \mathrm{span}\{\langle \nu,e_i \rangle: i=1,...,n+1\}$  of $C^{\infty}(\Sigma)$.
We know that $\dim E_1 = \dim E_2 = n+1$ [see \cite[Theorem $3.1$]{Perdomo}].
    Since $ E_1 \cap E_2 = \{0\}$, define $E = E_1 \oplus E_2$. Let us show that the symmetric bilinear form $\mathrm{Q}|_{E}$ is negative definite. Given $f = f_1 + f_2 \in  E - \{0\}$, with $f_1 \in E_1,\ f_2 \in E_2$, we have
  \begin{align}
   \nonumber    \mathrm{Q}(f) &= -\int_{\Sigma}fJf d\mu + \int_{\partial\Sigma}f(\nabla_{\eta}f - \mathrm{q}f)ds\\
     \nonumber  &= -\int_{\Sigma}fJf d\mu\\
     \nonumber  &= -\int_{\Sigma}(f_1 + f_2)J(f_1+f_2) d\mu\\
    \label{Eq. 3.8} &= -\int_{\Sigma}f_1Jf_1 d\mu -\int_{\Sigma}f_1Jf_2 d\mu -\int_{\Sigma}f_2Jf_1 d\mu -\int_{\Sigma}f_2Jf_2 d\mu.
  \end{align}
In the second equality we used Lemma \ref{Lemma 1.2}.
  
  Let us prove the following.
  
  \begin{Claim}\label{Claim 1}
      For any $f_1 \in E_1, f_2 \in E_2$, we have $-\int_{\Sigma}f_1 Jf_2 d\mu - \int_{\Sigma}f_2 Jf_1 d\mu = 0$.
  \end{Claim}
  \begin{proof}
      Since $\Sigma$ is a minimal hypersurface, Lemma \ref{Lemma 1.1} says that 
      \begin{align}\label{Eq. 3.9}
          Jf_1 = |\mathrm{A}|^2 f_1
      \end{align}
      \begin{align}\label{Eq. 3.10}
          Jf_2 = n f_2
      \end{align}
      By equations \eqref{Eq. 3.9}-\eqref{Eq. 3.10}, $|\mathrm{A}|$ and $n$ are constant eigenvalues of $J$. Since $|\mathrm{A}|^2 \neq n$
      \begin{align*}
          -\int_{\Sigma}f_1 Jf_2 d\mu - \int_{\Sigma}f_2 Jf_1 d\mu =  -\int_{\Sigma}  (|\mathrm{A}|^2 +n) f_1f_2 d\mu = 0,
          \end{align*}
          as  desired.
  \end{proof}
  Returning to the proof of the theorem, equations \eqref{Eq. 3.8}, \eqref{Eq. 3.9} and \eqref{Eq. 3.10} force 
  \begin{align*}
      \mathrm{Q}(f) &=  -\int_{\Sigma}f_1 Jf_1 d\mu  -\int_{\Sigma}f_2 Jf_2 d\mu\\
      &= -\int_{\Sigma}|\mathrm{A}|^2 f_1^2 d\mu  -n\int_{\Sigma} f_2^2 d\mu < 0.
  \end{align*}
  Thus $\mathrm{Q}|_{E}$ is negative definite, and therefore 
  \begin{align*}
      \mathrm{MI}(\Sigma) \geq \dim E = \dim(E_1\oplus E_2) = 2(n+1),
  \end{align*}
  proving $(iii)$.\\

  Now, if $\Sigma$ is  totally geodesic  the Jacobi operator reduces to $J = \Delta + n$, where $\Delta$ is the Laplacian operator on the unit sphere $\Sigma = \mathbb{S}^n_{+}$. The Neumann problem's eigenvalues equal $\lambda_i = \mu_i - n$, where $\mu_i$ denotes the i-th eigenvalue of the Laplacian on $\Sigma = \mathbb{S}^n_{+}$, with the same multiplicity. 

        The eigenvalues of the Laplacian on $\mathbb{S}^n_{+}$ are
        \begin{align*}
            \mu_i = (i-1)(n+i-2); i=1,2,...
        \end{align*}
        with multiplicities
        \begin{align*}
            m_{\mu_1} = 1, m_{\mu_2} = n
        \end{align*}
        and
        $m_{\mu_k}=$ $\begin{pmatrix}
          n+k-2 &\\
          k-2 
\end{pmatrix}$ - $\begin{pmatrix}
          n+k-4 &\\
          k-4 
\end{pmatrix}$, with $k=4,5,...$
\begin{itemize}
    \item [Case $1$.] If $i=1$ then $\mu_1 = 0$ and so $\lambda_1 = -n < 0$ with multiplicity $1$.
    \item [Case $2$.]If $i=2$ then $\mu_2 = n$, giving $\lambda_2 = 0$
    \item [Case $3$.] If $i=3$ then $\mu_3 = 2(n+1)$ and  $\lambda_3 = n+2 > 0$.
\end{itemize}
Therefore if $k \geq 2$ then $\lambda_k \geq 0$. 
Since $\Sigma$ is minimal, Lemma \ref{Lemma 1.1} implies $J\langle \nu, e_i \rangle = n \langle \nu, e_i \rangle$, for $e_i \in T\Sigma$ and $i=1,...,n+1$. Choose $\eta = e_{n+2}$. Lemma \ref{Lemma 1.2} tells 
\begin{align*}
    \nabla_{\eta}\langle \nu, e_i \rangle = \mathrm{A}(\eta, \eta)\langle \eta, e_i \rangle = 0.
\end{align*}
Hence the Neumann problem is satisfied.
Furthermore, $\mathrm{MI}(\Sigma) = 1$, proving $(i)$.\\

Every Clifford torus has $|\mathrm{A}|^2 = n$. Then the Jacobi operator reduces to $J = \Delta + 2n$, where $\Delta$ is the Laplacian operator of $\Sigma = \mathbb{S}^{k}_{+}(\sqrt{\frac{k}{n}}) \times \mathbb{S}^{n-k}(\sqrt{\frac{n-k}{n}})$. The eigenvalues of the Neumann problem are 
\begin{align*}
    \lambda_- = \mu_i - 2n,
\end{align*}
where $\mu_i$ is the i-th eigenvalue of $\Delta$. Hence the index of $\Sigma$ equals the number of eigenvalues of $\Delta$ counted with multiplicities.

  Let us to compute it. Suppose $\alpha$ is an eigenvalue of the Laplacian operator on a Riemannian manifold $M$ with multiplicity $m_{\alpha}$, and $\beta$ an eigenvalue of the Laplacian on a Riemannian manifold $N$ with multiplicity $m_{\beta}$. Then $\mu = \alpha + \beta$ is an eigenvalue of the Laplacian of the Riemannian product $M  \times N$, with multiplicity $m_{\alpha}m_{\beta}$.\\

  Notice that the eigenvalues of the Laplacian of $\mathbb{S}^{k}_{+}(\sqrt{\frac{k}{n}})$ are 
  \begin{align*}
      \alpha_i = \frac{n(i-1)(k+i-2)}{k}, i=1,2,3,...
  \end{align*}
  with multiplicities $m_{\alpha_1}=1$, $m_{\alpha_2}=k$,...
   $m_{\alpha_k}=$ $\begin{pmatrix}
          k+i-2 &\\
          i-2 
\end{pmatrix}$ - $\begin{pmatrix}
          k+i-4 &\\
          i-4 
\end{pmatrix}$, with $i=4,5,...$
and the Laplacian's eigenvalues on $\mathbb{S}^{n-k}(\sqrt{\frac{n-k}{n}})$ are given by
\begin{align*}
      \beta_j = \frac{k(j-1)(n-k+j-2)}{n-k}, j=1,2,3,...
  \end{align*}
  with multiplicities $m_{\beta_1}=1$, $m_{\beta_2}=k$,...
   $m_{\alpha_k}=$ $\begin{pmatrix}
          n-k+j-2 &\\
          j-2 
\end{pmatrix}$ - $\begin{pmatrix}
          n-k+j-4 &\\
          j-4 
\end{pmatrix}$, for $j=4,5,...$
\begin{itemize}
    \item [Case $1$.] If $i=1=j$ then $\mu_1 = 0$ and $\lambda_1 = -2n < 0$ with multiplicity $1$.
    \item [Case $2$.]If $i=2$ and $j=2$ then $\mu_2 = n$ and $\lambda_2 = -n$ with multiplicity $k$
    \item [Case $4$.] If $i=2$ and $j=1$ then $\mu_2 = n$ and $\lambda_2 = -n$ with multiplicity $n-k$.
    \item [Case $5$.] If $i=2$ and $j=2$ then $\mu_3 = 2n$ and $\lambda_3 = 0$.
\end{itemize}
Consequently if $k \geq 3$ then $\lambda_k \geq 0$. Moreover $\mathrm{MI}(\Sigma) = n+1$, proving $(ii)$.
\end{proof}

\subsection{A Sharp Estimate for the Eigenvalue of Jacobi Operator on Minimal Free-Boundary Hypersurfaces in $\mathbb{S}^{n+1}_{+}$}\label{Subsection 3.1}
In \cite{Perdomo} Perdomo showed that when $\lambda_1 = -2n$, $\Sigma$ is the minimal Clifford torus. To prove this theorem he  used the maximum principle. In \cite[Theorem $3$]{Alías 1} Alías gave an estimate of the first eigenvalue of  compact minimal hypersurfaces. Continuing along this thread, let us prove an upper bound for $\lambda_1$ in the following theorem.

\begin{theorem}
 Let $x: \Sigma \longrightarrow \mathbb{S}^{n+1}_{+}$ be a free-boundary minimal hypersurface such that $\nabla_{\eta}|\mathrm{A}|^2 = 0$. Then either
 \begin{enumerate}[label=(\roman*)]
 \item  $\lambda_1 \leq -2n $, with equality if and only if $\Sigma$ is \eqref{eq:1/2-Clifford} or
 \item $\lambda_1 = -n$ and $\Sigma$ is totally geodesic.
 \end{enumerate}
\end{theorem}
\begin{proof}
     In Simons' equation \cite{Simon} 
   \begin{align}\label{Eq. 1.26}
       \frac{1}{2}\Delta |\mathrm{A}|^2 = |\nabla \mathrm{A}|^2 + |\mathrm{A}|^2 (n- |\mathrm{A}|^2).
   \end{align}
the left-hand side equals 
   \begin{align}\label{Eq. 1.27}
       \frac{1}{2}\Delta |\mathrm{A}|^2 = |\mathrm{A}|\Delta |\mathrm{A}|+|\nabla |\mathrm{A}||^2.
   \end{align}
   By (\ref{Eq. 1.26})-(\ref{Eq. 1.27})
   \begin{align*}
       |\mathrm{A}|\Delta |\mathrm{A}| + |\nabla |\mathrm{A}||^2 &= |\nabla \mathrm{A}|^2 + |\mathrm{A}|^2 (n- |\mathrm{A}|^2).
   \end{align*}
   This implies that
   \begin{align}\label{Eq. 1.28}
       |\mathrm{A}|\Delta|\mathrm{A}| = |\nabla \mathrm{A}|^2 + n|\mathrm{A}|^2 - |\mathrm{A}|^4 - |\nabla |\mathrm{A}||^2.
   \end{align}

    For the next step we use the following auxiliary result, which can be found in \cite[Lemma $1$]{Alías 3}.
    \begin{lemma}\label{Lemma 1.6}
        Let $\Sigma$ be a Riemannian manifold. Consider a symmetric tensor $\mathrm{T}: \mathcal{X}(\Sigma) \longrightarrow \mathcal{X}(\Sigma)$ such that $\mathrm{tr}(\mathrm{T})=0$ and the covariant derivative $\nabla \mathrm{T}$ is symmetric. Then
        \begin{align*}
            |\nabla|\mathrm{T}|^2|^2 \leq \frac{4n}{n+2}|\mathrm{T}|^2|\nabla \mathrm{T}|^2.
        \end{align*}
    \end{lemma}
    Returning to the proof of the theorem, by equation (\ref{Eq. 1.28})
    \begin{align*}
        &|\mathrm{A}|\Delta |\mathrm{A}| + n|\mathrm{A}|^2 + |\mathrm{A}|^4 = |\nabla \mathrm{A}|^2 + 2n|\mathrm{A}|^2 - |\nabla |\mathrm{A}||^2\\
        &|\mathrm{A}|(\Delta|\mathrm{A}|+ |\mathrm{A}|^3 + n|\mathrm{A}|) = |\nabla \mathrm{A}|^2 + 2n|\mathrm{A}|^2 - |\nabla |\mathrm{A}||^2,
    \end{align*}
    and therefore
    \begin{align}\label{Eq. 1.29}
        |\mathrm{A}|J|\mathrm{A}| = |\nabla \mathrm{A}|^2 + 2n|\mathrm{A}|^2 - |\nabla |\mathrm{A}||^2.
    \end{align}
    Notice that $\nabla |\mathrm{A}|^2 = 2|\mathrm{A}|\nabla |\mathrm{A}|$ and $|\nabla |\mathrm{A}|^2|^2 = 4 |\mathrm{A}|^2 |\nabla |\mathrm{A}||^2$. Take $|\mathrm{T}|^2 = |\mathrm{A}|^2$, so Lemma \ref{Lemma 1.6} implies
    \begin{align}\label{Eq 1.30}
      |\nabla |\mathrm{A}|^2|^2 = 4 |\mathrm{A}|^2 |\nabla |\mathrm{A}||^2  \leq \frac{4n}{n+2}|\mathrm{A}|^2|\nabla \mathrm{A}|^2  
    \end{align}
 and then
    \begin{align}\label{Eq 1.31}
       |\nabla |\mathrm{A}|^2|^2 \leq \frac{n}{n+2}|\nabla \mathrm{A}|^2. 
    \end{align}
    Using (\ref{Eq 1.31}) 
    \begin{align*}
        |\mathrm{A}|J |\mathrm{A}| &\geq |\nabla \mathrm{A}|^2 - \frac{n}{n+2}|\nabla \mathrm{A}|^2 + 2n|\mathrm{A}|^2\\
        &= \frac{2}{n+2}|\nabla \mathrm{A}|^2 + 2n|\mathrm{A}|^2,
    \end{align*}
    so
    \begin{align*}
       - |\mathrm{A}|J|\mathrm{A}| &\leq -\frac{2}{n+2}|\nabla \mathrm{A}|^2 - 2n|\mathrm{A}|^2\\
       & \leq -2n||\mathrm{A}|^2.
    \end{align*}
    Note that $|\mathrm{A}| \neq 0$ and $\lambda_1 \leq \dfrac{-\int_{\Sigma}|\mathrm{A}|J |\mathrm{A}|d\mu+\int_{\partial \Sigma}|\mathrm{A}|(\nabla_{\eta}|\mathrm{A}|-q|\mathrm{A}|^2)ds}{\int_{\Sigma}|\mathrm{A}|d\mu}$. 
    
   Since $q = 0$ and $|\mathrm{A}|\nabla_{\eta}|\mathrm{A}| = \frac{1}{2}\nabla_{\eta}|\mathrm{A}|^2$ we obtain 
    \begin{align*}
        \lambda_1 \leq -\dfrac{\int_{\Sigma}|\mathrm{A}|J |\mathrm{A}|d\mu}{\int_{\Sigma}|\mathrm{A}|^2d\mu}.
    \end{align*}
    Therefore
    \begin{align*}
        \lambda_1 \leq -\frac{2}{n+2}\dfrac{\int_{\Sigma}|\nabla \mathrm{A}|^2 d\mu}{\int_{\Sigma}|\mathrm{A}|^2d\mu} - 2n.
    \end{align*}
     In particular, since $|\nabla \mathrm{A}|^2$ and $|\mathrm{A}|^2$ are non-negative, $\lambda_1 \leq -2n$, proving the first part of $(i)$. Let us prove the last part of $(i)$.\\

     We have
\begin{align}\label{Eq. 1.31}
    \lambda_1 \leq -2n - \frac{2}{n+2}|\nabla \mathrm{A}|^2 \leq -2n.
\end{align}
    If $\lambda_1 = -2n$, by (\ref{Eq. 1.31}),
    \begin{align*}
        |\nabla \mathrm{A}| = 0.
    \end{align*}
    By Lemma \ref{Lemma 1.6}, $0 \leq |\nabla |\mathrm{A}||^2 \leq \frac{4n}{n+2}|\nabla \mathrm{A}|^2 = 0$, and so 
    \begin{align}\label{Eq. 1.32}
        |\nabla |\mathrm{A}|| = 0.
    \end{align}
    Hence by \eqref{Eq. 1.32} $|\mathrm{A}|$ is non-negative and  constant. Since $|\mathrm{A}| \neq 0$ then $|\mathrm{A}|$ is positive and constant.
    Moreover, $J|\mathrm{A}| = \Delta|\mathrm{A}| + |\mathrm{A}|^3 +n|\mathrm{A}| = |\mathrm{A}|^3 +n|\mathrm{A}| = |\mathrm{A}|(|\mathrm{A}|^2 + n)$ and $|\mathrm{A}|^2 +n$ is constant. Therefore
    \begin{align*}
        -2n =\lambda_1 = -(|\mathrm{A}|^2 +n).
    \end{align*}
    This implies $|\mathrm{A}|^2 = n$.  Theorem \ref{Theorem 3.1} now says  $\Sigma$ is the upper Clifford torus.\\

    Finally, take the test function $f=1$ in 
    \begin{align*}
        \lambda_1 = \inf \left\{\frac{-\int_{\Sigma}fJf d\mu + \int_{\partial \Sigma}f(\nabla_{\eta}f - \mathrm{q}f)ds}{\int_{\Sigma}f^2 d\mu}\right\}.
    \end{align*}
    Since $\nabla_{\eta}f = 0$ and $\mathrm{q}=0$ then
    \begin{align*}
        \lambda \leq \frac{Q(1)}{\int_{\Sigma} d\mu} = \frac{Q(1)}{\mathrm{Area}(\Sigma)}= -n - \frac{1}{\mathrm{Area}(\Sigma)}\int_{\Sigma}|\mathrm{A}|^2 d\mu.
    \end{align*}
    Using that $|\mathrm{A}|$ is non-negative we obtain
    \begin{align*}
         \lambda \leq  -n - \frac{1}{\mathrm{Area}(\Sigma)}\int_{\Sigma}|\mathrm{A}|^2 d\mu \leq -n.
    \end{align*}
    Futhermore, $\lambda_1 = -n$ iff $|\mathrm{A}| = 0$, proving $(ii)$.
    
\end{proof}

\section{Free-boundary $\mathrm{CMC}$ hypersurfaces on $\mathbb{S}^{n+1}_{+}$}\label{Section 4}

Now we take $x: \Sigma \longrightarrow \mathbb{S}^{n+1}_{+}$ to be a free-boundary hypersurface with constant mean curvature $\mathrm{H}>0$. Similarly to what was done in the minimal case, we will study the Morse index. In Theorem \ref{Theorem 3.1} we proved that either the Morse index is equal to one (if $\Sigma = \mathbb{S}^{n}_{+}(r)$ is totally umbilical with $0<r<1$) or greater than or equal to $n+1$ if $\Sigma$ is not totally umbilical. It equals $n+1$ when $\Sigma$ is the $\mathrm{H}$-torus. \\

As consequence of the first variation for the area, $x: \Sigma \longrightarrow \mathbb{S}^{n+1}_{+}$ has constant mean curvature (not necessarily zero)
 iff $\frac{d}{dt}\mathcal{A}_{f}(t) = 0$ for every smooth function $f \in C^{\infty}(\Sigma)$ such that $\int_{\Sigma}f d\mu = 0$. Let us assume that $\frac{d}{dt}\mathcal{A}_{f}(t) = 0$ for every smooth function $f \in C^{\infty}(\Sigma)$ such that $\int_{\Sigma}f d\mu = 0$ and notice that $\mathrm{H} = \mathrm{H}_0 + (\mathrm{H} - \mathrm{H}_0)$, where $\mathrm{H}_0 = \int_{\Sigma}\frac{1}{\mathrm{Area}(\Sigma)}\mathrm{H}d\Sigma$.\\

 If we consider $f = \mathrm{H} - \mathrm{H}_0$, since $\int_{\Sigma}(\mathrm{H} - \mathrm{H}_0)d\mu = 0$, then
 \begin{align*}
     0= \frac{d}{dt}\mathcal{A}_{f}(t) &= \frac{d}{dt}\mathcal{A}_{\mathrm{H} - \mathrm{H}_0}(t)\\
     &= -n\int_{\Sigma}(\mathrm{H} - \mathrm{H}_0)\mathrm{H}d\mu + \int_{\partial \Sigma}\langle Y, \eta \rangle ds.
 \end{align*}
 This is implies $\mathrm{H} = \mathrm{H}_0$ is constant on $\Sigma$ and $Y \perp \eta$. Conversely, if  $\mathrm{H} = \mathrm{H}_0$ is constant on $\Sigma$ and $Y \perp \eta$ then $\frac{d}{dt}\mathcal{A}_{f}(t) = 0$.  In  other words free-boundary CMC hypersurfaces are critical for the area functional under volume-preserving variations.\\

 There exist two different notions of index:  the strong index  $\mathrm{MI}(\Sigma)$ associated with the Neumann problem, and the weak index  $\mathrm{MI}^{W}(\Sigma)$ associated with the  Neumann problem. The former is simply
 \begin{align*}
     \mathrm{MI}(\Sigma) = \dim \max\{\ V \geq C^{\infty}(\Sigma): Q(f) < 0; f \in V, f \neq 0\},
 \end{align*}
 and $x: \Sigma \longrightarrow \mathbb{S}^{n+1}_{+}$ is called stable if $ \mathrm{MI}(\Sigma) = 0$. On the other hand, the weak index is defined by
 \begin{align*}
      \mathrm{MI}^{W}(\Sigma) = \dim \max \{V \geq C^{\infty}_{W}(\Sigma): Q(f) < 0; f \in V, f \neq 0\},
 \end{align*}
 where $C^{\infty}_{W}(\Sigma) = \{f \in C^{\infty}(\Sigma): \int_{\Sigma}f d\mu = 0 \}$. The hypersurface $x: \Sigma \longrightarrow \mathbb{S}^{n+1}_{+}$ is called weakly stable if $ \mathrm{MI}^{W}(\Sigma) = 0$.\\

 By the min-max principle
 \begin{align*}
     \lambda_1 \leq \lambda_1^{W} \leq \lambda_2 \leq \lambda_2^{W}  \leq ...,
 \end{align*}
 were we denote by $\lambda_i^{W}$ the  weak eigenvalue of the Neumann problem. From a geometrical perspective, the weak index is more natural than the strong index.

\subsection{Morse Index for free-boundary Hypersurfaces with Constant Mean Curvature on  $\mathbb{S}^{n+1}_{+}$}\label{Subsection 4.1}
\hspace{0.5cm} If $x: \Sigma \longrightarrow \mathbb{S}^{n+1}_{+}$ is a totally umbilical hypersurface, it is more convenient to work with the traceless second fundamental form $\mathrm{\circA} = \mathrm{A} - n\mathrm{H}I$, where $I$ is the identity operator on $\mathcal{X}(\Sigma)$. Notice that $\mathrm{tr}|\mathrm{\circA}| = 0$ and $|\mathrm{\circA}|^2 = |\mathrm{A}|^2 - \frac{\mathrm{H}^2}{n} \geq 0$, with equality iff $x: \Sigma \longrightarrow \mathbb{S}^{n+1}_{+}$ is totally umbilical. In terms of $\mathrm{\circA}$, we can write the Jacobi operator as
\begin{align*}
    J = \Delta + |\mathrm{\circA}|^2 + n\left(1 + \frac{\mathrm{H}^2}{n^2} \right).
\end{align*}

For compact hypersurfaces into $\mathbb{S}^{n+1}$ Alías proved \cite[Theorem $9$]{Alías 1} that the weak Morse index of a  non-totally umbilical hypersurface is at least $n+2$, with equality if and only if $\Sigma$ is the $\mathrm{H}$-Clifford torus. He also  showed that the weak Morse index of a totally umbilical hypersurface is exactly $0$.
We shall discuss a  generalization of \cite[Theorem $8$]{Alías 1} and \cite[Theorem $9$]{Alías 1} regarding the weak Morse index of free-boundary hypersurfaces with constant mean curvature $\mathrm{H}>0$ in $\mathbb{S}^{n+1}_{+}$.

\begin{theorem}\label{Theorem 3.5}
     Let $x: \Sigma \longrightarrow \mathbb{S}^{n+1}_{+}$ be a free-boundary hypersurface with constant mean curvature $\mathrm{H}>0$, $\eta = e_{n+2}$ a conormal unit vector of $\partial \Sigma$. If $|\mathrm{A}|$ is constant, $x: \Sigma \longrightarrow \mathbb{S}^{n+1}_{+}$ is weakly stable if and only if $x: \Sigma \longrightarrow \mathbb{S}^{n+1}_{+}$ is a totally umbilical $\mathbb{S}^n_{+} \subset \mathbb{S}^{n+1}_{+}$.
\end{theorem}
\begin{proof}
    Let $x: \Sigma \longrightarrow \mathbb{S}^{n+1}_{+}$ be a weakly stable CMC hypersurface, that is, 
    \begin{align*}
        Q(f) \geq 0,
    \end{align*}
    for every smooth function $f \in C^{\infty}(\Sigma)$ such that $\int_{\Sigma}f d\mu = 0$. Since $\mathrm{H}$ is constant, we have $|\mathrm{A}|^2 = |\mathrm{\circA}|^2 + \frac{\mathrm{H}^2}{n}$. Then Lemma \ref{Lemma 1.1} forces
    \begin{align*}
        \Delta \langle \nu,e_i \rangle &= -|\mathrm{A}|^2   \langle \nu,e_i \rangle + \mathrm{H} \langle x,e_i \rangle\\
        &= -\left(|\mathrm{\circA}|^2 + \frac{\mathrm{H}^2}{n}\right)   \langle \nu,e_i \rangle + \mathrm{H} \langle x,e_i \rangle\\
        &= -|\mathrm{\circA}|^2 \langle \nu, e_i \rangle - \mathrm{H}\left(\frac{\mathrm{H}}{n}\langle \nu, e_i \rangle - \langle x, e_i \rangle\right).
    \end{align*}
    Let us set $g_{e_i} = \frac{\mathrm{H}}{n}\langle \nu, e_i \rangle - \langle x, e_i \rangle$. By Lemma \ref{Lemma 1.1},
    \begin{align*}
        \int_{\Sigma}g_{e_i}d\mu &= \int_{\Sigma}\left(\frac{\mathrm{H}}{n}\langle \nu, e_i \rangle - \langle x, e_i \rangle \right)d\mu\\
        &= \frac{1}{n}\left(\frac{\mathrm{H}}{n}\int_{\Sigma}\langle \nu, e_i \rangle d\mu  + \frac{1}{n}\int_{\Sigma} \Delta \langle x,e_i \rangle d\mu - \frac{\mathrm{H}}{n}\int_{\Sigma}\langle \nu, e_i \rangle\right).
    \end{align*}
    Stokes' theorem and Lemma \ref{Lemma 1.2} imply
    \begin{align*}
         \int_{\Sigma}g_{e_i}d\mu &= \frac{1}{n}\left(\frac{\mathrm{H}}{n}\int_{\Sigma}\langle \nu, e_i \rangle d\mu  + \frac{1}{n}\int_{\Sigma} \Delta \langle x,e_i \rangle d\mu - \frac{\mathrm{H}}{n}\int_{\Sigma}\langle \nu, e_i \rangle d\mu \right)\\
         &=  \frac{1}{n^2}\int_{\partial \Sigma} \nabla_{\eta} \langle x,e_i \rangle ds\\
         &= \frac{1}{n^2}\int_{\partial \Sigma} \langle \eta,e_i \rangle ds = 0,
    \end{align*}
    where the last equality is due to $\eta = e_{n+2}$. Hence 
    \begin{align*}
         \int_{\Sigma}g_{e_i}d\mu = 0.
    \end{align*}
    Now, let us compute $J g_{e_i}$. By Lemma \ref{Lemma 1.1}
    \begin{align*}
        J g_{e_i} &= \Delta g_{e_i} + |\mathrm{\circA}|^2 g_{e_i} + n\left( 1+ \frac{\mathrm{H}}{n^2}\right) g_{e_i}\\
        &= - \frac{\mathrm{H}|\mathrm{\circA}|^2}{n}\langle \nu, e_i \rangle + |\mathrm{\circA}|^2 \left( \frac{\mathrm{H}}{n}\langle \nu, e_i \rangle - \langle x, e_i \rangle \right)\\
        &= - |\mathrm{\circA}|^2 \langle x, e_i \rangle.
    \end{align*}
    We then obtain
    \begin{align*}
        0 \leq \sum_{i=1}^{n+1}Q(g_{e_i}) &= \sum_{i=1}^{n+1}(\int_{\Sigma}g_{e_i} J g_{e_i} d\mu + \int_{\partial \Sigma}g_{e_i}(\nabla_{\eta}g_{e_i} - \mathrm{q}g_{e_i})ds)\\
        &= -\frac{|\mathrm{\circA}|^2 \mathrm{H}}{n}\sum_{i=1}^{n+1}\int_{\Sigma}\langle x,e_i \rangle \langle \nu, e_i \rangle d\mu - |\mathrm{\circA}|^2\sum_{i=1}^{n+1}\int_{\Sigma}\langle x,e_i \rangle ^2 d\mu + \sum_{i=1}^{n+1}\int_{\partial \Sigma}g_{e_i}(\nabla_{\eta}g_{e_i} - \mathrm{q}g_{e_i})ds).
    \end{align*}
    But $x: \Sigma \longrightarrow \mathbb{S}^{n+1}_{+}$ is free-boundary, so   $\mathrm{q} = 0$ and by Lemma \ref{Lemma 1.2} we have $\nabla_{\eta}g_{e_i} = 0$. Hence
    \begin{align*}
         0 \leq \sum_{i=1}^{n+1}Q(g_{e_i}) &= -\frac{|\mathrm{\circA}|^2 \mathrm{H}}{n}\sum_{i=1}^{n+1}\int_{\Sigma}\langle x,e_i \rangle \langle \nu, e_i \rangle d\mu - |\mathrm{\circA}|^2\sum_{i=1}^{n+1}\int_{\Sigma}\langle x,e_i \rangle ^2 d\mu + \sum_{i=1}^{n+1}\int_{\partial \Sigma}g_{e_i}(\nabla_{\eta}g_{e_i} - \mathrm{q}g_{e_i})ds)\\
         &= -\frac{|\mathrm{\circA}|^2 \mathrm{H}}{n}\sum_{i=1}^{n+1}\int_{\Sigma}\langle x,e_i \rangle \langle \nu, e_i \rangle d\mu - |\mathrm{\circA}|^2\sum_{i=1}^{n+1}\int_{\Sigma}\langle x,e_i \rangle ^2 d\mu. 
    \end{align*}
    Moreover, free-boundary also implies  $\sum_{i=1}^{n+1}\langle \nu, e_i \rangle \langle x,e_i \rangle = \langle \nu, x \rangle = 0$ and we know $\sum_{i=1}^{n+1}\langle x, e_i \rangle ^2 = \langle x, x \rangle = 1$. Moreover, 
    \begin{align*}
        0 \leq \sum_{i=1}^{n+1}Q(g_{e_i}) &= -\frac{|\mathrm{\circA}|^2 \mathrm{H}}{n}\sum_{i=1}^{n+1}\int_{\Sigma}\langle x,e_i \rangle \langle \nu, e_i \rangle d\mu - |\mathrm{\circA}|^2\sum_{i=1}^{n+1}\int_{\Sigma}\langle x,e_i \rangle ^2 d\mu\\
        &= -|\mathrm{\circA}|^2 \sum_{i=1}^{n+1}\int_{\Sigma}\langle x, e_i \rangle ^2 d\mu\\
        &= -|\mathrm{\circA}|^2\mathrm{Area}(\Sigma) \leq 0.
    \end{align*}
    We conclude that $|\mathrm{\circA}| = 0$, i.e. $x: \Sigma \longrightarrow \mathbb{S}^{n+1}_{+}$ is totally umbilical.\\

    Conversely, a totally umbilical hypersurface in $\mathbb{S}^{n+1}_{+}$ must be $\mathbb{S}^{n}_{+}(r)$; $0<r<1$. Since $\mathrm{H} = \frac{n \sqrt{1-r^2}}{r}$ then $1+\frac{\mathrm{H}^2}{n^2} = \frac{1}{r^2}$. The Jacobi operator reduces to $J = \Delta + \frac{n}{r^2}$, where $\Delta$ is the Laplacian operator on $\mathbb{S}^{n}_{+}(r)$, with $0<r<1$. Hence the eigenvalues of $J$ are $\lambda_i = \mu_i - \frac{n}{r^2}$, where $\mu_i$ denotes the i-th eigenvalue of the Laplacian on $\Sigma = \mathbb{S}^n_{+}(r)$, with the same multiplicity.
         The eigenvalues of the Laplacian on $\mathbb{S}^n_{+}(r)$ are
        \begin{align*}
            \mu_i = \frac{(i-1)(n+i-2)}{r^2}; i=1,2,...
        \end{align*}
        with multiplicities
        \begin{align*}
            m_{\mu_1} = 1, m_{\mu_2} = n
        \end{align*}
        and
        $m_{\mu_k}=$ $\begin{pmatrix}
          n+k-2 &\\
          k-2 
\end{pmatrix}$ - $\begin{pmatrix}
          n+k-4 &\\
          k-4 
\end{pmatrix}$, with $k=4,5,...$\\

In particular, $\lambda_1 = -\frac{n}{r^2}<0$ has multiplicity $1$ and is associated with  constant eigenfunctions. Since all other eigenfunctions of $J$ (for the Neumann problem) are orthogonal to the constants, they satisfy $\int_{\Sigma}f d\mu = 0$ and thus fulfil Neumann's boundary condition.\\
Hence, 
\begin{align*}
    \lambda_i^{W} = \lambda_i = \mu_{i+1} - \frac{n}{r^2},
\end{align*}
for $i \geq 1$. Since $\mu_2 = \frac{n}{r^2}$ then $\lambda_1^{W} = \lambda_2 = \mu_2 - \frac{n}{r^2} = 0$. It follow that $x: \Sigma \longrightarrow \mathbb{S}^{n+1}_{+}$ is weakly stable, and the proof ends.
\end{proof}

\begin{theorem}\label{Theorem 3.6}
    Let $x: \Sigma \longrightarrow \mathbb{S}^{n+1}_{+}$ be a free-boundary hypersurface with constant mean curvature $\mathrm{H}>0$ and $\eta = e_{n+2}$ a conormal unit vector of $\partial \Sigma$. If $|\mathrm{A}|$ is constant  then 
    \begin{itemize}
        \item i. either $\mathrm{MI}_{W}(\Sigma) = 0$ and $ \mathbb{S}^{n}_{+}(r) \subset \mathbb{S}^{n+1}_{+}$,
        \item ii. or $\mathrm{MI}_{W}(\Sigma) \geq n+1$, with equality if and only if $\Sigma$ is the upper $\mathrm{H}$- Clifford torus $\mathbb{S}^{k}(r)_{+}\times \mathbb{S}^{n-k}(\sqrt{1-r^2})$ with radius $\sqrt{\frac{k}{n+2}} \leq r \leq \sqrt{\frac{k+2}{n+2}}$. 
    \end{itemize}
    
\end{theorem}
\begin{proof}
Let $x: \Sigma \longrightarrow \mathbb{S}^{n+1}_{+}$ be totally umbilical. By Theorem \ref{Theorem 3.5} we have $x$ is weakly stable, that is, $\mathrm{MI}^{W}(\Sigma) = 0$, proving $(i)$.\\

Suppose $\Sigma$ is not totally umbilical. By Lemma \ref{Lemma 1.1}
    
    \begin{align*} 
    J \begin{pmatrix}  \langle x,e_i \rangle \\  \langle\nu, e_i \rangle \end{pmatrix} =   \begin{pmatrix}  |\mathrm{A}|^2 & \mathrm{H} \\  \mathrm{H} & n \end{pmatrix}\begin{pmatrix}  \langle x,e_i \rangle \\  \langle\nu, e_i \rangle \end{pmatrix}.
    \end{align*}

The characteristic polynomial 
\begin{align*}
    p(\lambda) = \begin{vmatrix}
    |\mathrm{A}|^2 - \lambda & \mathrm{H} \\ \mathrm{H} & n -\lambda
    \end{vmatrix} =  \lambda^2 - ( |\mathrm{A}|^2 + n)\lambda + n|\mathrm{A}|^2- \mathrm{H}^2
\end{align*}
has roots 
\begin{align*}
    \lambda_{-} = \dfrac{|\mathrm{A}|^2 + n -  { \sqrt{\delta}}}{2}\\
     \lambda_{+} = \dfrac{|\mathrm{A}|^2 + n +  \sqrt{\delta}}{2},
\end{align*}
where  $\delta = \sqrt{(|\mathrm{A}|^2 - n)^2 + 4\mathrm{H}^2}$.
Observe that $\lambda_{\pm}$ are distinct and real since $|\mathrm{A}|^2 \neq n$ and $\delta > 0$.
  Then $Jf - \lambda_{\pm}f = 0$, that is, $\lambda_{\pm}$ are eigenvalues of J.\\

Set $\gamma_{\pm} = \dfrac{n -\lambda_{\pm}+\mathrm{H}}{|\mathrm{A}|^2 -\lambda_{\pm}+\mathrm{H}}$. Consider the subspace $E_{+} = \mathrm{span}\{\langle x,e_i \rangle + \gamma_{+}\langle \nu, e_i \rangle;  i=1,...,n+1\}$.
For each $E_{\pm}$, we have
\begin{align*}
    J(\langle x,e_i \rangle + \gamma_{\pm}\langle \nu, e_i \rangle) = \lambda_{\pm}(\langle x,e_i \rangle + \gamma_{\pm}\langle \nu, e_i \rangle).
\end{align*}
Let $\psi_{\pm}: \mathbb{R}^{n+1} \longrightarrow \mathrm{E}_{\pm}$ be the linear surjective map $\psi_{\pm}(e_i) = \langle x,e_i \rangle + \gamma_{\pm}\langle \nu, e_i \rangle$. By construction  $\mathrm{Im}\psi_{\pm} =  E_{\pm}$. The rank-nullity theorem implies 
\begin{align*}
    &n+1 = \dim \mathbb{R}^{n+1} = \dim \mathrm{Ker}\psi_{+} + \dim E_{+}\\
    &n+1 = \dim \mathbb{R}^{n+1} = \dim \mathrm{Ker}\psi_{-} + \dim E_{-}
\end{align*}
and so 
\begin{align}\label{Eq. 3.2}
    2(n+2) = \dim \mathrm{Ker}\psi_{+} + \dim \mathrm{Ker}\psi_{-} + \dim E_{+} + \dim E_{-}.
\end{align}

 \begin{Claim}\label{Claim 2}
        $\mathrm{Ker}\psi_{+} \cap \mathrm{Ker}\psi_{-} = \{0\}$.
 \end{Claim}
 \begin{proof}
     Assume there exists a vector $v \in \mathrm{Ker}\psi_{+} \cap \mathrm{Ker}\psi_{-}$. Then
     \begin{align*}
         \langle x,v \rangle + \gamma_{+}\langle \nu, v \rangle = \langle x,v \rangle + \gamma_{-}\langle \nu, v \rangle.
     \end{align*}
    This implies that
     \begin{align*}
         (\gamma_{+}-\gamma_{-})(\langle \nu, v \rangle).
     \end{align*}
     Using the fact that $\gamma_{+}-\gamma_{-} > 0$ we deduce $\langle \nu, v \rangle = 0$. That is, $\Sigma$ is a totally geodesic equator of $\mathbb{S}^{n+1}_{+}$ [see \cite[Theorem $1$]{Alías 2}], which contradicts the hypothesis. Moreover $\mathrm{Ker}\psi_{+} \cap \mathrm{Ker}\psi_{-} = \{0\}$, as desired.
 \end{proof}
 By Claim \ref{Claim 2} together with $\dim(\mathrm{Ker}\psi_{+} \oplus \mathrm{Ker}\psi_{-}) \leq n+1$ we obtain 
 \begin{align*}
     n+1 \geq \dim (ker \psi_{+}\oplus ker \psi_{-}) = \dim ker \psi_{+} + \dim ker \psi_{-} - \dim (ker \psi_{+} \cap ker \psi_{-}) =  \dim ker \psi_{+} + \dim ker \psi_{-}.
 \end{align*}
 That is,
 \begin{align}\label{Eq. 3.3}
     \dim ker \psi_{+} + \dim ker \psi_{-} \leq n+1.
 \end{align}
Equations (\ref{Eq. 3.2})-(\ref{Eq. 3.3}) force
 \begin{align}\label{Eq. 3.4}
     \dim E_{+} + \dim E_{-} \geq n+1.
 \end{align}

 Let us prove $Q|_{E_{-} \oplus E_{+}}$ is negative definite. For $f_1 \in E_{-}$ and $f_2 \in E_{+}$ such that $f=f_1 + f_2$, 
\begin{align*}
    Q(f) &= -\int_{\Sigma}fJf d\mu + \int_{\partial\Sigma}f(\nabla_{\eta}f - \mathrm{q}f)ds\\
    &=  -\int_{\Sigma}f_1Jf_1 d\mu -  \int_{\Sigma}f_2Jf_2 d\mu -  \int_{\Sigma}f_2Jf_1 d\mu &-\\ \\ & \  \ -  \int_{\Sigma}f_1Jf_2 d\mu+\int_{\partial\Sigma}(f_1 + f_2)(\nabla_{\eta}f_1 + \nabla_{\eta}f_2  - \mathrm{q}(f_1 + f_2))ds.
\end{align*}
Since $Jf-\lambda_{\pm}f=0$ with $\lambda_{\pm}$ distinct, $E_{+}$ and $E_{-}$ are orthogonal. Then 
\begin{align*}
    Q(f) &= -\int_{\Sigma}fJf d\mu + \int_{\partial\Sigma}f(\nabla_{\eta}f - \mathrm{q}f)ds\\
    &=  -\int_{\Sigma}\lambda_{-}f_1^2 d\mu -  \int_{\Sigma}\lambda_{+}f_2^2 d\mu +   &+\\ &+ \int_{\partial\Sigma}(f_1 + f_2)(\nabla_{\eta}f_1 + \nabla_{\eta}f_2)ds  - \int_{\partial\Sigma}\mathrm{q}(f_1^2 + f_2^2)ds.
\end{align*}
Since $e_i \perp \eta$, for $i=1,...,n+1$, where $\eta = e_{n+2}$,  by Lemma (\ref{Lemma 1.2}) $\nabla_{\eta}f_1 = \nabla_{\eta}f_2 = 0$. But $\Sigma$ is free-boundary, so $\mathrm{q} = 0$. Hence
\begin{align}\label{Eq. 2.5}
    Q(f) &=  -\int_{\Sigma}f_1Jf_1 d\mu -  \int_{\Sigma}f_2Jf_2 d\mu \\
    &= -\int_{\Sigma}\lambda_{-}f_1^2 d\mu -  \int_{\Sigma}\lambda_{+}f_2^2 d\mu < 0.
\end{align}

Furthermore $Q|_{E_{-} \oplus E_{+}}$ is negative definite. By (\ref{Eq. 3.4})
\begin{align*}
    \mathrm{MI}(\Sigma) \geq \dim(E_{-} \oplus E_{+}) = \dim E_{-} + \dim E_{+} \geq n+1,
\end{align*}
proving the first part of $(ii)$. Now, let us to prove the last part.\\

Let $x: \Sigma \longrightarrow \mathbb{S}^{n+1}_{+}$ be now the $\mathrm{H}$-Clifford torus. The Jacobi operator reduces to 
\begin{align*}
    J = \Delta + \frac{k}{r^2} + \frac{n-k}{1-r^2},
\end{align*}
 where $\Delta$ is the Laplacian on $\Sigma = \mathbb{S}^{k}_{+}(r) \times \mathbb{S}^{n-k}(\sqrt{1-r^2})$. In particular, $\lambda_1 = -\left(\frac{k}{r^2} + \frac{n-k}{1-r^2} \right)<0$, with multiplicity $1$,  has  constant eigenfunctions. Since all eigenfunctions of $J$ (for the Neumann problem) are orthogonal to the constants, they satisfy $\int_{\Sigma}f d\mu = 0$ and the Neumann boundary condition holds, $\nabla_{\eta}f = 0$. Also, $\Sigma$  is free-boundary so $\mathrm{q} = 0$. Similarly to  Theorem \ref{Theorem 3.5}, we have
 \begin{align*}
     \lambda_i^{W} = \lambda_{i+1} = \mu_{i+1} - \left(\frac{k}{r^2} + \frac{n-k}{1-r^2} \right),
 \end{align*}
 for $i \geq 1$. Counting eigenvalues as we did in Theorem \ref{Theorem 3.2} we obtain 
 \begin{align*}
     \lambda_1^{W} = \lambda_2 = -\frac{k}{r^2}, 
 \end{align*}
 with multiplicity $k$.
 \begin{align*}
     \lambda_1^{W} = \lambda_2 = -\frac{n-k}{1-r^2} < 0, 
 \end{align*}
 with multiplicity $n-k+1$.
 Moreover $\lambda_2^{W} = 0$, and for $i \geq 2$ we have $\lambda_i^{W} \geq 0$. Hence $\mathrm{MI}^{W}(\Sigma) = n+1$. 
\end{proof}

\subsection{A sharp estimate for the first eigenvalue of the Jacobi operator of free-boundary CMC hypersurfaces in $\mathbb{S}^{n+1}_{+}$}\label{Subsection 4.2}

\hspace{0.5cm} We proved in Lemma \ref{Lemma 3.1} that $\Sigma$ is either totally geodesic or half of a Clifford torus. The next theorem will provide, under certain conditions, a classification of hypersurfaces with constant mean curvature $\mathrm{H}>0$ in  $\mathbb{S}^{n+1}_{+}$. The techniques are borrowed from \cite{Alencar}. In terms of $\mathrm{\circA}$, from $|\mathrm{\circA}|^2 = |\mathrm{A}|^2 - \frac{\mathrm{H}^2}{n}$ the Jacobi operator is $J = \Delta + |\mathrm{\circA}|^2 + n\left( 1 + \frac{\mathrm{H}^2}{n^2}\right)$. In Theorem \ref{Theorem 2.2} we used a generalization  of Simons' equation and other tools from  \cite[Theorem $1.5$]{Alencar}.

\begin{theorem}\label{Theorem 2.2}
    Let $x: \Sigma \longrightarrow \mathbb{S}^{n+1}_{+}$ be a free-boundary hypersurface with constant mean curvature $\mathrm{H}>0$ and assume  $\nabla_{\eta}|\mathrm{\circA}|^2 = 0$.  Then either $\Sigma$ is totally umbilical or $|\mathrm{\circA}|= \alpha_{\mathrm{H}}$, where 
    \begin{align*}
        \alpha_{\mathrm{H}} = \frac{-(n-2)\mathrm{H}+\sqrt{(n-2)\mathrm{H}^2 + 4n(n-1)(1+\mathrm{H}^2)}}{2\sqrt{n-1}}
    \end{align*}
    is the positive root of the polynomial
    \begin{align*}
        P_{\mathrm{H}}(x)= x^2 + \frac{(n-2)}{\sqrt{n-1}}\mathrm{H}x - n\left(1+\frac{\mathrm{H}^2}{n^2}\right).
    \end{align*}
\end{theorem}
\begin{proof}
    The extension of Simon's equation for CMC hypersurfaces reads
    \begin{align}\label{Eq. 2.6}
        \frac{1}{2}\Delta|\mathrm{\circA}|^2 = |\nabla \mathrm{\circA}|^2 + \left(n\left(1+\frac{\mathrm{H}^2}{n^2}\right)-|\mathrm{\circA}|^2\right)|\mathrm{\circA}|^2 + \mathrm{H}\mathrm{tr}(\mathrm{\circA^3}).  
        \end{align}
        We will use the following  result \cite[Lemma $2.6$]{Alencar}
        \begin{lemma}\label{Lema 2.3}
            Let $a_1,...,a_n$ be real numbers such that $\sum_{i=1}^{n}a_i = 0$. Then
            \begin{align*}
                -\frac{n-2}{\sqrt{n(n-1)}}(\sum_{i=1}^{n}a^2_i)^{\frac{3}{2}}\leq \sum_{i=1}^{n}a^3_i \leq \frac{n-2}{\sqrt{n(n-1)}}(\sum_{i=1}^{n}a^2_i)^{\frac{3}{2}}.
            \end{align*}
            Moreover, equality holds on the right (or the left) if and only if $(n-1)$ of the $a_i$  are non-positive (respectively, non-negative) and equal.
        \end{lemma}
        Since $\mathrm{tr}(\mathrm{\circA})=0$, by Lemma \ref{Lema 2.3}
         \begin{align*}
            |\mathrm{tr}(\mathrm{\circA^3})| \leq \frac{n-2}{\sqrt{n(n-1)}}|\mathrm{\circA}|^3 
        \end{align*}
        and so 
        \begin{align}\label{Eq. 2.7}
         -\frac{(n-2)}{\sqrt{n(n-1)}}\mathrm{H}|\mathrm{\circA}|^3   \leq \mathrm{H}\mathrm{tr}(\mathrm{\circA^3}) \leq  \frac{(n-2)}{\sqrt{n(n-1)}}\mathrm{H}|\mathrm{\circA}|^3.
        \end{align}
        By equations (\ref{Eq. 2.6}) and (\ref{Eq. 2.7})
        \begin{align*}
            \frac{1}{2}\Delta|\mathrm{\circA}|^2 \geq |\nabla \mathrm{\circA}|^2 - |\mathrm{\circA}|^2\left(\frac{(n-2)\mathrm{H}|\mathrm{\circA}|}{\sqrt{n(n-1)}}+ |\mathrm{\circA}|^2 - n\left(1+\frac{\mathrm{H}^2}{n^2}\right)\right).
        \end{align*}
        Put $P_{\mathrm{H}}(|\mathrm{\circA}|) = \frac{(n-2)\mathrm{H}|\mathrm{\circA}|}{\sqrt{n(n-1)}}+ |\mathrm{\circA}|^2 - n\left(1+\frac{\mathrm{H}^2}{n^2}\right)$, so that 
        \begin{align}\label{Eq. 2.8}
             \frac{1}{2}\Delta|\mathrm{\circA}|^2 \geq |\nabla \mathrm{\circA}|^2 - |\mathrm{\circA}|^2 P_{\mathrm{H}}(|\mathrm{\circA}|).
        \end{align}
     Integrating equation (\ref{Eq. 2.8}), 
     \begin{align*}
         \frac{1}{2}\int_{\Sigma}\Delta|\mathrm{\circA}|^2 d\mu \geq \int_{\Sigma}|\nabla \mathrm{\circA}|^2 d\mu - \int_{\Sigma}|\mathrm{\circA}|^2 P_{\mathrm{H}}(|\mathrm{\circA}|)d\mu.
     \end{align*}
     Since $\nabla_{\eta}|\mathrm{\circA}|^2 = 0$, we use Stokes' to obtain
     \begin{align}\label{Eq. 2.9}
         0 \geq \int_{\Sigma}|\nabla \mathrm{\circA}|^2d\mu - \int_{\Sigma}|\mathrm{\circA}|^2 P_{\mathrm{H}}(|\mathrm{\circA}|)d\mu.
     \end{align}
    Using that $\int_{\Sigma}|\mathrm{\circA}|^2 \geq 0$ and (\ref{Eq. 2.9}), we find 
     \begin{align*}
         0 \leq \int_{\Sigma}|\nabla \mathrm{\circA}|^2 d\mu \leq \int_{\Sigma}|\mathrm{\circA}|^2 P_{\mathrm{H}}(|\mathrm{\circA}|)d\mu.
     \end{align*}
     Over $[0, \alpha_{\mathrm{H}}]$ the expression  $P_{\mathrm{H}}(|\mathrm{\circA}|)$ is non-positive, so
      \begin{align}\label{Eq. 2.10}
         0 \leq \int_{\Sigma}|\nabla \mathrm{\circA}|^2 d\mu \leq \int_{\Sigma}|\mathrm{\circA}|^2 P_{\mathrm{H}}(|\mathrm{\circA}|)d\mu  \leq 0.
     \end{align}
     Therefore by (\ref{Eq. 2.10}) 
     \begin{align*}
         \int_{\Sigma}|\nabla \mathrm{\circA}|^2 d\mu = 0.
     \end{align*}
     Since $|\nabla \mathrm{\circA}|^2$ is non-negative, we conclude $\nabla \mathrm{\circA} = 0$. On the other hand,  $\mathrm{H}$ is constant, and then
     \begin{align*}
         0 = \nabla \mathrm{A} = \nabla \mathrm{\circA}.
     \end{align*}
Therefore $\Sigma$ has exactly two principal curvatures \cite[Lemma $1$]{Lawson} with multiplicity $n-1$ and $1$. Since $\Sigma$ is isoparametric, by \cite[Theorem $4.20$]{Dacjer} it is contained in a hypersurface of revolution  of $\mathbb{S}^{n+1}$. By \cite[Theorem $1.2$]{Alencar}, $\Sigma$ is the free-boundary $\mathrm{H}$-Clifford torus.
         Moreover, by (\ref{Eq. 2.10}) $\int_{\Sigma}|\mathrm{\circA}|^2 P_{\mathrm{H}}(|\mathrm{\circA}|)d\mu = 0$, and so 
         \begin{align*}
             \textrm{either}\  \ |\mathrm{\circA}| = 0  \  \ \textrm{or} \  \  P_{\mathrm{H}}(|\mathrm{\circA}|) = 0.
         \end{align*}
         That is, either $\Sigma$ is totally umbilical or $|\mathrm{\circA}| = \alpha_{\mathrm{H}}$.
\end{proof}

\hspace{0.5cm}The proof of Theorem \ref{Theorem 4.4} relies on the classical characterization of $\lambda_1$ and has the advantage of providing us with both a
sharp inequality for $\lambda_1$ and the characterization of equality. We wish to show how Perdomo’s method, based on a
maximum principle, also works to characterize the equality case (item $(ii)$ of Theorem \ref{Theorem 4.4}). 
This is interesting because, beside providing
an alternative argument for equality, it serves to better understand other
similar problems, as well as free-boundary  CMC hypersuperfaces. To prove the following theorem we will use ideas taken from \cite[Proposition $3.1$]{Alías 4}.

\begin{theorem}\label{Theorem 4.4}
    Let $x: \Sigma \longrightarrow \mathbb{S}^{n+1}_{+}$ be a free-boundary hypersurface with constant mean curvature $\mathrm{H}>0$ and $\eta = e_{n+2}$ a conormal unit vector of $\partial \Sigma$. If $|\mathrm{A}|$ is constant then either
    \begin{enumerate}[label=(\roman*)]
    \item  $\lambda_1 = -n\left( 1+\dfrac{\mathrm{H}}{n^2}\right)$, and  $\Sigma = \mathbb{S}^n_{+}(r) \subset \mathbb{S}^{n+1}_{+}$ is totally umbilical with $0<r<1$, or 
        \item  $\lambda_1 \leq -2n\left( 1+\dfrac{\mathrm{H}}{n^2}\right) + \dfrac{\mathrm{H}(n-2)}{\sqrt{n(n-1)}}\frac{\int_{\Sigma}|\mathrm{\circA}|^3d\mu}{\int_{\Sigma}|\mathrm{\circA}|^2d\mu}$, with equality if and only if $\Sigma$ is half of the $\mathrm{H}$-torus.
        \end{enumerate}
    \end{theorem}
    
\begin{proof}
    Suppose $x: \Sigma \longrightarrow \mathbb{S}^{n+1}_{+}$ is not totally umbilical: $|\mathrm{\circA}| \neq 0$. The generalized Simons equation gives 
    \begin{align}\label{Eq. 3.12}
        \frac{1}{2}\Delta|\mathrm{\circA}|^2 = |\nabla \mathrm{\circA}|^2 + \left(n\left(1+\frac{\mathrm{H}^2}{n^2}\right)-|\mathrm{\circA}|^2\right)|\mathrm{\circA}|^2 + \mathrm{H}\mathrm{tr}(\mathrm{\circA^3}).  
        \end{align}
       On the other hand
       \begin{align}\label{Eq. 3.13}
           \Delta|\mathrm{\circA}|^2 = |\mathrm{\circA}|\Delta |\mathrm{\circA}| + |\nabla |\mathrm{\circA}||^2.
       \end{align}
By \eqref{Eq. 3.12}, \eqref{Eq. 3.13} and \eqref{Eq. 2.7}
      \begin{align}\label{Eq. 3.14}
            |\mathrm{\circA}|\Delta |\mathrm{\circA}| + |\nabla |\mathrm{\circA}||^2 \geq |\nabla\mathrm{\circA}|^2 + n\left(1+\frac{\mathrm{H}^2}{n^2}\right)|\mathrm{\circA}|^2 - |\mathrm{\circA}|^4 + \frac{(n-2)}{\sqrt{n(n-1)}}\mathrm{H}|\mathrm{\circA}|^3. 
       \end{align}
       Using Lemma \ref{Lema 2.3}, \eqref{Eq. 2.8} and \eqref{Eq. 3.14} we can add  the term $n\left(1+\frac{\mathrm{H}^2}{n^2}\right)|\mathrm{\circA}|^2$ in the last inequality
       \begin{align}\label{Eq. 3.15}
         |\mathrm{\circA}|\Delta |\mathrm{\circA}| + |\mathrm{\circA}|^4 +  n\left(1+\frac{\mathrm{H}^2}{n^2}\right)|\mathrm{\circA}|^2 \geq \frac{2}{n+2}|\nabla \mathrm{\circA}|^2 +2n\left(1+\frac{\mathrm{H}^2}{n^2}\right)|\mathrm{\circA}|^2 - \frac{(n-2)}{\sqrt{n(n-1)}}\mathrm{H}|\mathrm{\circA}|^3.
       \end{align}
        Since $J = \Delta + |\mathrm{\circA}|^2 + n\left(1+\frac{\mathrm{H}^2}{n^2}\right)$, by \eqref{Eq. 3.15}
       \begin{align}\label{Eq. 3.16}
           |\mathrm{\circA}|J|\mathrm{\circA}| \geq \frac{2}{n+2}|\nabla \mathrm{\circA}|^2 +2n\left(1+\frac{\mathrm{H}^2}{n^2}\right)|\mathrm{\circA}|^2 - \frac{(n-2)}{\sqrt{n(n-1)}}\mathrm{H}|\mathrm{\circA}|^3
       \end{align}
       Notice that $|\mathrm{\circA}| \neq 0$ and $\lambda_1 \leq \dfrac{-\int_{\Sigma}|\mathrm{\circA}|J|\mathrm{\circA}|d\mu + \int_{\partial \Sigma}|\mathrm{\circA}|(\nabla_{\eta}|\mathrm{\circA}|-\mathrm{q}|\mathrm{\circA}|)ds}{\int_{\Sigma}|\mathrm{\circA}|^2 d\mu}$. Hence, since $\nabla_{\eta}|\mathrm{\circA}|^2 = 2 |\mathrm{\circA}|\nabla_{\eta}|\mathrm{\circA}|$ we can rewrite the inequality for $\lambda_1$ as
    \begin{align}\label{Eq. 3.17}
    \lambda_1 \leq \dfrac{-\int_{\Sigma}|\mathrm{\circA}|J|\mathrm{\circA}|d\mu + \frac{1}{2}\int_{\partial \Sigma}|\nabla_{\eta}|\mathrm{\circA}|^2 ds -\int_{\partial \Sigma}\mathrm{q}|\mathrm{\circA}|^2ds}{\int_{\Sigma}|\mathrm{\circA}|^2d\mu}.
\end{align}
We have assumed $\nabla_{\eta}|\mathrm{\circA}|^2 = 0$, so \eqref{Eq. 3.17} gives 
\begin{align}\label{Eq. 3.18}
    \lambda_1 \leq \dfrac{-\int_{\Sigma}|\mathrm{\circA}|J|\mathrm{\circA}|d\mu  -\int_{\partial \Sigma}\mathrm{q}|\mathrm{\circA}|^2ds}{\int_{\Sigma}|\mathrm{\circA}|^2d\mu}.
\end{align}
By \eqref{Eq. 3.16} and \eqref{Eq. 3.18}
\begin{align*}
     \lambda_1 &\leq \dfrac{-\int_{\Sigma}|\mathrm{\circA}|J|\mathrm{\circA}|d\mu  -\int_{\partial \Sigma}\mathrm{q}|\mathrm{\circA}|^2ds}{\int_{\Sigma}|\mathrm{\circA}|d\mu}\\
     & \leq -\frac{2}{n+2}\frac{\int_{\Sigma}|\nabla \mathrm{\circA}|^2 d\mu}{\int_{\Sigma}|\mathrm{\circA}|^2d\mu} - 2n\left(1+\frac{\mathrm{H}^2}{n^2}\right) +  \frac{\mathrm{H}(n-2)}{\sqrt{n(n-1)}}\frac{\int_{\Sigma}|\mathrm{\circA}|^3 d\mu}{\int_{\Sigma}|\mathrm{\circA}|^2 d\mu} - \mathrm{q}\frac{\int_{\partial\Sigma}|\mathrm{\circA}|^2 ds}{\int_{\Sigma}|\mathrm{\circA}|^2 d\mu}.
\end{align*}
Since $\mathrm{q} = 0$
\begin{align*}
     \lambda_1 \leq -\frac{2}{n+2}\frac{\int_{\Sigma}|\nabla \mathrm{\circA}|^2 d\mu}{\int_{\Sigma}|\mathrm{\circA}|^2d\mu} - 2n\left(1+\frac{\mathrm{H}^2}{n^2}\right) +  \frac{\mathrm{H}(n-2)}{\sqrt{n(n-1)}}\frac{\int_{\Sigma}|\mathrm{\circA}|^3 d\mu}{\int_{\Sigma}|\mathrm{\circA}|^2 d\mu}.
\end{align*}
   In particular, since $|\nabla \mathrm{\circA}|^2$ and $|\mathrm{\circA}|^2$ are non-negative, then 
   \begin{align*}
        \lambda_1 \leq  - 2n\left(1+\frac{\mathrm{H}^2}{n^2}\right) +  \frac{\mathrm{H}(n-2)}{\sqrt{n(n-1)}}\frac{\int_{\Sigma}|\mathrm{\circA}|^3 d\mu}{\int_{\Sigma}|\mathrm{\circA}|^2 d\mu},
   \end{align*}
   proving $(ii)$.\\

   Suppose that $\lambda_1 = - 2n\left(1+\frac{\mathrm{H}^2}{n^2}\right) +  \frac{\mathrm{H}(n-2)}{\sqrt{n(n-1)}}\frac{\int_{\Sigma}|\mathrm{\circA}|^3 d\mu}{\int_{\Sigma}|\mathrm{\circA}|^2 d\mu}$. By item $(i)$
     \begin{align}\label{Eq. 3.19}
        0 \leq  -\frac{2}{n+2}\frac{\int_{\Sigma}|\nabla \mathrm{\circA}|^2 d\mu}{\int_{\Sigma}|\mathrm{\circA}|^2d\mu} \leq 0.
    \end{align}
    Since $|\nabla \mathrm{\circA}|^2$ and $|\mathrm{\circA}|$ are non-negative, \eqref{Eq. 3.19} implies $\nabla \mathrm{\circA} = 0$. Consequently $|\nabla |\mathrm{\circA}||^2 = 0$. Hence $|\mathrm{\circA}|$ is constant and non-negative. Since $|\mathrm{\circA}|\neq 0$, it follows that  $|\mathrm{\circA}|$ is a positive constant.      Furthermore
    \begin{align*}
        J|\mathrm{\circA}| &= \Delta |\mathrm{\circA}| + |\mathrm{\circA}|^3 + n\left(1+\frac{\mathrm{H}^2}{n^2}\right) |\mathrm{\circA}|\\
        &=  |\mathrm{\circA}|\left(|\mathrm{\circA}|^2 + n\left(1+\frac{\mathrm{H}^2}{n^2}\right)\right),
    \end{align*}
    where $-\lambda_1 = |\mathrm{\circA}|^2 + n\left(1+\frac{\mathrm{H}^2}{n^2}\right)$.
    Also,
     \begin{align*}
        \lambda_1 = - 2n\left(1+\frac{\mathrm{H}^2}{n^2}\right) +  \frac{\mathrm{H}(n-2)}{\sqrt{n(n-1)}}\max_{\Sigma}|\mathrm{\circA}|.
    \end{align*}
    Hence
     \begin{align*}
        |\mathrm{\circA}|^2  = n\left(1+\frac{\mathrm{H}^2}{n^2}\right) -  \frac{\mathrm{H}(n-2)}{\sqrt{n(n-1)}}\max_{\Sigma}|\mathrm{\circA}|, 
    \end{align*}
 and from \cite[Theorem $1.1$]{Alías 4} $\Sigma$ is a free-boundary $\mathrm{H}-$torus.\\

Finally, take the test function $f=1$ in
      \begin{align*}
        \lambda_1 = \inf \left\{\frac{-\int_{\Sigma}fJf d\mu + \int_{\partial\Sigma}f(\nabla_{\eta}f-\mathrm{q}f)ds}{\int_{\Sigma}f^2 d\mu}; f \neq 0; f \in C^{\infty}(\Sigma) \right\}.
    \end{align*}
    Since $\nabla_{\eta}f = 0$ and $\mathrm{q} = 0$, then
    \begin{align*}
        \lambda_1 \leq \frac{Q(1)}{\int_{\Sigma}d\mu} = \frac{Q(1)}{\mathrm{Area}({\Sigma})}= -n\left( 1+\frac{\mathrm{H}}{n^2}\right) - \frac{1}{\mathrm{Area}(\Sigma)}\int_{\Sigma}|\mathrm{\circA}|^2 d\mu.
    \end{align*}
    Since $|\mathrm{\circA}|^2$ is non-negative, 
    \begin{align*}
        \lambda_1 \leq  -n\left( 1+\frac{\mathrm{H}}{n^2}\right) - \frac{1}{\mathrm{Area}(\Sigma)}\int_{\Sigma}|\mathrm{\circA}|^2 d\mu \leq -n\left( 1+\frac{\mathrm{H}}{n^2}\right).
    \end{align*}
    Furthermore, $\lambda_1 = -n\left( 1+\frac{\mathrm{H}}{n^2}\right)$ if and only if $|\mathrm{\circA}| = 0$, proving $(i)$.
\end{proof}

\vspace{1cm}

\noindent CRÍSIA DE OLIVEIRA \\
Instituto de Matem\'atica e Estat\'istica\\
Universidade Federal Fluminense\\
Rua professor Marcos Waldemar de Freitas Reis s/n (campus Gragoatá)\\ 24210-201 Niterói/RJ \\ 
Brazil\\
{\it email}:\ crisiaoliveira@id.uff.br

\end{document}